\documentclass[%
 aip,
cha,
 amsmath,amssymb,
 reprint,%
]{revtex4-1}

\usepackage{graphicx}
\usepackage{dcolumn}
\usepackage{bm}

\usepackage[utf8]{inputenc}
\usepackage[T1]{fontenc}
\usepackage{mathptmx}
\usepackage{etoolbox}

\usepackage{amsthm} 
\usepackage[hidelinks]{hyperref} 
\usepackage[capitalise]{cleveref} 
\crefname{paragraph}{paragraph}{paragraphs} 

\usepackage{scalerel}
\usepackage{tikz}
\usetikzlibrary{svg.path}

\definecolor{orcidlogocol}{HTML}{A6CE39}
\tikzset{
  orcidlogo/.pic={
    \fill[orcidlogocol] svg{M256,128c0,70.7-57.3,128-128,128C57.3,256,0,198.7,0,128C0,57.3,57.3,0,128,0C198.7,0,256,57.3,256,128z};
    \fill[white] svg{M86.3,186.2H70.9V79.1h15.4v48.4V186.2z}
                 svg{M108.9,79.1h41.6c39.6,0,57,28.3,57,53.6c0,27.5-21.5,53.6-56.8,53.6h-41.8V79.1z M124.3,172.4h24.5c34.9,0,42.9-26.5,42.9-39.7c0-21.5-13.7-39.7-43.7-39.7h-23.7V172.4z}
                 svg{M88.7,56.8c0,5.5-4.5,10.1-10.1,10.1c-5.6,0-10.1-4.6-10.1-10.1c0-5.6,4.5-10.1,10.1-10.1C84.2,46.7,88.7,51.3,88.7,56.8z};
  }
}

\newcommand\orcidicon[1]{\href{https://orcid.org/#1}{\mbox{\scalerel*{
\begin{tikzpicture}[yscale=-1,transform shape]
\pic{orcidlogo};
\end{tikzpicture}
}{|}}}}

\theoremstyle{definition}
\newtheorem{Definition}{Definition}
\newtheorem{Example}{Example}
\theoremstyle{plain}
\newtheorem{Lemma}{Lemma}
\newtheorem{Proposition}{Proposition}
\newtheorem*{Proposition*}{Proposition}
\newtheorem{Corollary}{Corollary}
\theoremstyle{remark}
\newtheorem{Remark}{Remark}

\makeatletter
\def\@email#1#2{%
 \endgroup
 \patchcmd{\titleblock@produce}
  {\frontmatter@RRAPformat}
  {\frontmatter@RRAPformat{\produce@RRAP{*#1\href{mailto:#2}{#2}}}\frontmatter@RRAPformat}
  {}{}
}%
\makeatother

\begin{document}

\preprint{AIP/123-QED}

\title[Machine learning of Lagrangian densities]{Learning of discrete models of variational PDEs from data}
\author{Christian Offen ${\protect \orcidicon{0000-0002-5940-8057}}$}
\email{christian.offen@uni-paderborn.de}
\homepage{https://www.uni-paderborn.de/en/person/85279}
\author{Sina Ober-Blöbaum}%
\affiliation{ 
Paderborn University, Department of Mathematics, Warburger Str. 100, 33098 Paderborn, Germany
}%

\date{14 November 2023}

\begin{abstract}
We show how to learn discrete field theories from observational data of fields on a space-time lattice. 
For this, we train a neural network model of a discrete Lagrangian density such that the discrete Euler--Lagrange equations are consistent with the given training data. We, thus, obtain a structure-preserving machine learning architecture.
Lagrangian densities are not uniquely defined by the solutions of a field theory. 
We introduce a technique to derive regularisers for the training process which optimise numerical regularity of the discrete field theory.
Minimisation of the regularisers guarantees that close to the training data the discrete field theory behaves robust and efficient when used in numerical simulations.
Further, we show how to identify structurally simple solutions of the underlying continuous field theory such as travelling waves. This is possible even when travelling waves are not present in the training data. This is compared to data-driven model order reduction based approaches, which struggle to identify suitable latent spaces containing structurally simple solutions when these are not present in the training data.
Ideas are demonstrated on examples based on the wave equation and the Schrödinger equation.
\\

Citation of published version\cite{DLNNPDE}: Chaos 34 (1), 013104 (2024). \url{https://doi.org/10.1063/5.0172287}.
\end{abstract}

\def\d{\mathrm{d}}
\def\D{\mathrm{D}}
\def\p{\partial}
\def\R{\mathbb{R}}
\def\N{\mathbb{N}}
\def\Z{\mathbb{Z}}
\def\vrtx{{{\mathrm{vrtx}}}}

\maketitle

\begin{quotation}
To learn practical models of dynamical systems from data, prior geometric knowledge and aspects of numerical integration theory can be exploited in the design of the data-driven architectures.
The article shows how to utilise numerical analysis informed neural networks and discrete field theories to model dynamical systems governed by first principles.

\end{quotation}

\section{Introduction}

The identification of models of dynamical systems from observational data is an important task in machine learning: the identified models can be used to interpolate data points, predict motions, develop control strategies, or analyse long-term behaviour\cite{Narendra1990,Han2019}. Additionally, physical properties of the system can be discovered such as symmetries, conservation laws\cite{SymHNN,SymLNN}, stress tensors\cite{Wu2018}, for instance. Moreover, data-driven models can be used as surrogates for analytic models when computations with explicit analytical models are prohibitively expensive, such as in molecular dynamics\cite{KADUPITIYA202010}.

To improve qualitative aspects of data-driven models, prior knowledge about physical aspects of the dynamics, such as symmetries, conservation laws, Hamiltonian or variational structure can be taken into account. The approach is referred to as {\em physics informed machine learning}. 
(Notice, however, that the term {\em physics informed neural networks} has also been used to describe neural network approximations of solutions to partial differential equations that represent physical laws. See Karniadakis et al\cite{Karniadakis2021} for a review article.)


Variational principles are at the heart of physical theories and constitute first principles from which dynamical models can be derived.
In such dynamical systems motions or solutions to the field equations constitute stationary points of an action functional. For instance, classical mechanical systems, electromagnetic field theories, fluid dynamics, as well as quantum dynamical systems are governed by variational principles.
Presence of variational structure in a dynamical system is related to many profound laws in physics such as the correspondence of symmetries and conserved quantities (Noether's theorem) \cite{Marsden1999,mansfield2010}.

To embed variational structure into machine learned models of dynamical systems, Greydanus et al propose to learn a model of a variational principle from data, called Lagrangian neural networks (LNNs) \cite{LNN}. This needs to be contrasted to learning of a model of a flow map or a vector field of a dynamical system without enforcing physical properties of the system \cite{Hamzi2021}. Similar ideas and extensions to LNN have been developed, such as combinations with model order reduction techniques \cite{Blanchette2020} or the inclusion of external forces \cite{Lutter2019}. The data-driven learning of variational principles needs to be contrasted to approaches that identify the analytic form of a (partial) differential equation via sparse regression \cite{Schaeffer2017,Kutz2017,tripura2023bayesian}. While the former aims at learning a model that can be used in numerical predictions, the latter focuses on interpretable system identification and requires a dictionary of candidate terms that can make up the differential equation.

Other approaches focus on exploiting Hamiltonian structure for learning models of dynamical systems. While Hamiltonian and variational structure are equivalent from a theoretical perspective (under non-degeneracy conditions), data-driven architectures for learning Hamiltonian systems, such as Hamiltonian Neural Networks\cite{HNN} or Symplectic Neural Networks \cite{SympNets}, differ from those architectures that learn models of action functionals as they require additional observation of momentum data rather than just position data (and possibly derivatives), or at least some prior knowledge of the symplectic phase space structure. However, much of the symplectic structure can be learned from data as well \cite{Bertalan2019,chen2023NeuralSymplecticForm}.

Instead of learning a model of a continuous action functional to describe a dynamical system with variational structure, Qin\cite{Qin2020} proposes to learn a discrete variational principle instead.
Discrete field theories can be trained on discrete observations of fields over a space-time lattice, while the training of continuous theories such as LNN\cite{LNN} requires observations of velocities and acceleration data.
Moreover, solutions of discrete field theories can be computed without further discretisation by numerical integrators such that discretisation errors are avoided.

\begin{figure}
\includegraphics[width=0.32\linewidth]{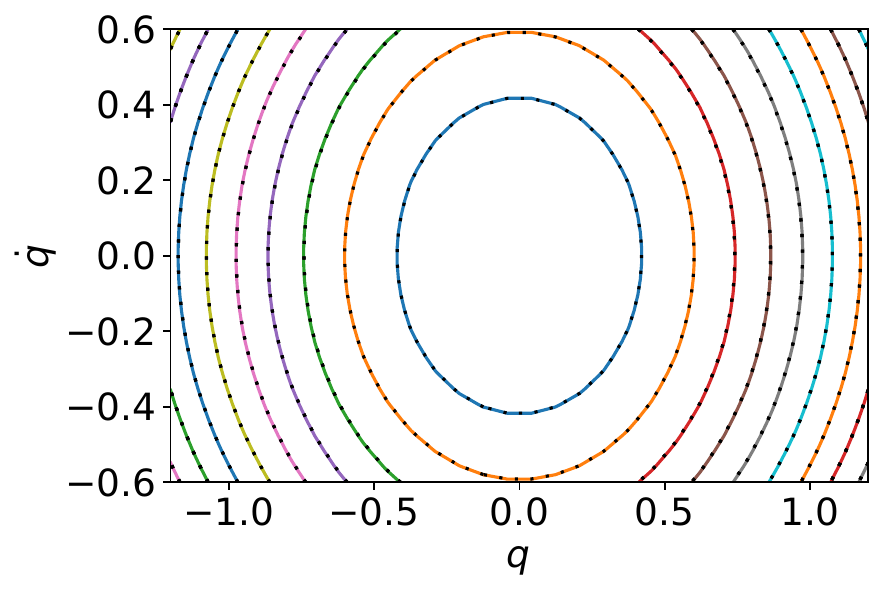}
\includegraphics[width=0.32\linewidth]{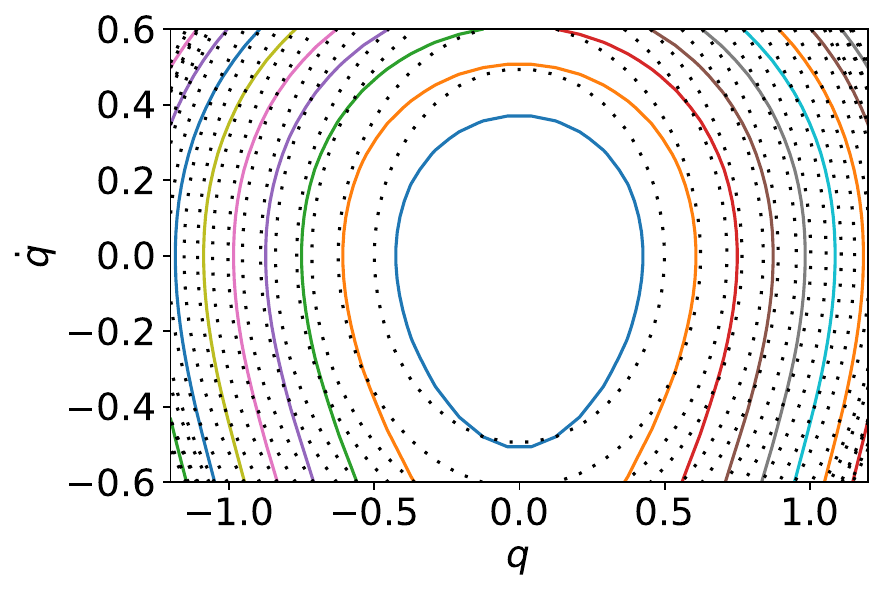}
\includegraphics[width=0.32\linewidth]{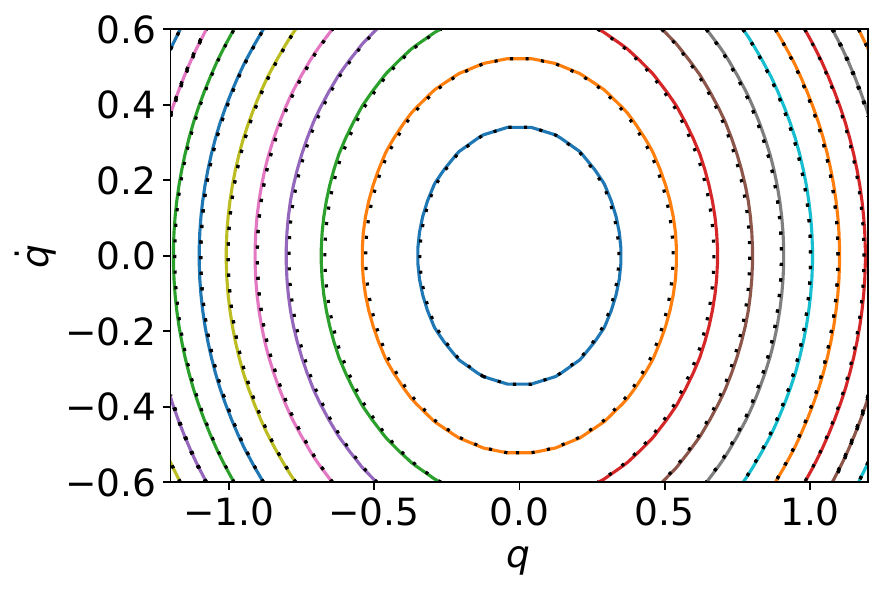}
\caption{Left: Motions of a data-driven model (solid lines) for the mathematical pendulum match motions of an analytic model (dotted). Centre: Numerically computed motions of the data-driven model are very inaccurate. Right: Applying the same numerical integrator to the analytic model yield much more accurate results.  (Figure taken from \cite[Fig.\ 6]{LagrangianShadowIntegrators}. See \cite{LagrangianShadowIntegrators} for details.)}\label{fig:1}
\end{figure}
However, the authors have demonstrated that numerical discretisation errors of learned continuous variational models of dynamical systems can be surprisingly big\cite{LagrangianShadowIntegrators}: the plot to the left of \cref{fig:1} shows that the exact trajectories of the learned dynamical system of the mathematical pendulum coincide with the exact trajectories of an analytic reference.
However, an application of a standard variational numerical integrator (variational midpoint rule\cite{MarsdenWestVariationalIntegrators}) to the learned model performs poorly (centre) while an application of the same integrator with the same step-size shows only small numerical errors when applied to the analytic reference (right).
The reason is that variational principles are not uniquely determined by a system's motions. In machine learning, care is required to identify a model of a variational principle that is not only consistent with the motion data, but also suitable for numerical computations. This will be a central aspect of this article.

While learning a discrete field theory avoids discretisation errors in numerical computation, additional regularisation is essential to avoid learning degenerate principles for which discrete field equations become unsolvable or badly conditioned. 
While the authors have discussed a geometric regularisation strategy applying to models based on Gaussian Processes or kernel methods\cite{LagrangianShadowIntegrators}, different regularisation strategies are required when artificial neural networks are used.
For variational dynamical systems, for which the underlying continuous system is an ordinary differential equation, we show first approaches in Lishkova et al\cite{SymLNN}. The present article extends the authors' conference contribution\cite{DLNNDensity}, in which machine learning architectures for discrete field theories are developed whose underlying dynamical system is a partial differential equation.

Machine learning of field theories governing infinite-dimensional dynamical systems needs to be contrasted with approaches that learn differential operators or a response function and do not exploit geometric structure or a discrete framework, such as DeepONet\cite{DeepONet} or pde-net\cite{PDENet,PDENet2}. Approaches that exploit Hamiltonian or Poisson structure can be found in some recent articles \cite{Matsubara2020,Jin2023,Eidnes2023}.

Further, identification of discrete field theories for partial differential equations needs to be contrasted with approaches based on model order reduction (MOR). Data-driven MOR-based techniques learn a reduction map of a spatially discretised system. Then, a full order model (FOM) is projected to the latent space to obtain a reduced order model (ROM). If no FOM is available, the learned reduction map is used to project dynamical data to the latent space. Subsequently, the ROM is learned, i.e.\ a data-driven model is fitted to describe the temporal dynamics of the system on the latent space.
Several approaches have been developed to preserve geometric structure in this context\cite{Blanchette2020,Blanchette2022,Glas2023,Carlsberg2015,sharma2023symplectic,sharma2022Lagrangian,Sharma2022,Tyranowski2022}.

In contrast to the data-driven MOR-based techniques, we learn the discrete field theory directly on a space-time lattice without a reduction step.
Indeed, discrete field theories are based on stencils coupling spacial and temporal variables only locally. 
We exploit this locality to obtain a reduction of computational complexity for training, which would otherwise require MOR. (Notice that computing solutions of the trained model could still be done by MOR-based techniques, when locality of the stencils cannot be exploited due to boundary conditions or numerical stability issues.)


We demonstrate that avoiding projection of the dynamics to a latent space can be advantageous if, for instance, highly symmetric solutions of the model, such as travelling waves, need to be identified in the data-driven model. We develop a method to detect and preserve solutions of the discrete field theory that correspond to highly symmetric solutions in the underlying continuous dynamical system. As the mesh of a discrete field theory is typically not compatible with the symmetries of the continuous system, we use a suitable symmetric ansatz of the discrete action functional based on Palais' principle of symmetric criticality \cite{palais1979,Torre2011} to design machine learning architectures for discrete field theories with travelling waves.

The main novelties of the article are
\begin{itemize}
	
	\item the systematic design of numerical analysis informed regularisation strategies for machine learning models of discrete field theories,
	
	\item the preservation and detection of highly symmetric solutions in discrete field theories using data-driven architectures based on Palais' principle of symmetric criticality and comparison to data-driven model order reduction based techniques,
	
	\item extension of the approaches first introduced in our conference contribution\cite{DLNNDensity} to different stencils and to degenerate Lagrangians that are linear in velocities.

\end{itemize}

The remainder of the article is structured as follows: \Cref{sec:Background} reviews the concept of variational principles in the continuous and discrete setting. The theory is illustrated on our main examples, the wave equation and the Schrödinger equation. 
\Cref{sec:RegPrepare} provides theoretical considerations motivating the regularisers that we will employ when training our neural network models.  \Cref{sec:LossFuns} introduces the machine learning set-up to learn a discrete field theory from data.
\Cref{sec:PSC} provides theoretical background on the notion of travelling waves and highly symmetric solutions in discrete field theories and describes their variational structure. 
Finally, \cref{sec:Experiments} provides numerical experiments based on the wave and Schrödinger equation.
The article concludes with a summary and future work section (\cref{sec:Summary}).

\section{Continuous and discrete variational principles and field theories}\label{sec:Background}

We recall briefly the concept of variational principles and their discretisations. 
We illustrate the theory on the wave equation and the Schrödinger equation, on which we will test our data-driven framework.
For more detailed expositions, we refer to the literature for an introduction to variational calculus\cite{Olver1986,Basdevant2007,mansfield2010} and discrete mechanics\cite{MarsdenWestVariationalIntegrators}.

\subsection{Continuous variational principles}
Variational principles play a fundamental role in the derivation and analysis of field theories in physics\cite{Basdevant2007}: motions $u\colon X \to \R^d$ are described as stationary points of an action functional 
\begin{equation}\label{eq:ActionFunctionalS}
S(u) = \int_X L\left(x,u(x),u_{x_0}(x),\ldots,u_{x_n}(x)\right) \d x.
\end{equation}
Here the domain of definition of $u$, $X$, refers to a suitable manifold (with boundary) such as $X=[x^0_0,x^N_0] \times \ldots \times [x^0_n,x^N_n]$ with coordinates $x=(x_0,x_1,\ldots,x_n)$.
Often, $X$ represents a space-time domain.
The expression $u_{x_k}$ denotes the partial derivative $\frac{\p u}{\p x_k}$. 
More precisely, $u$ constitutes a solution of the field theory, if the action functional $S$ is stationary at $u$ with respect to variations $\delta u \colon X \to \R^d$ vanishing at the boundary of $X$. The function $L$ is called {\em Lagrangian density}. As it depends on the 1-jet of $u$ only and not on higher derivatives, a field theory described by an action functional of the from \eqref{eq:ActionFunctionalS} is called {\em 1st order field theory}.
Stationary points of the action functional are solutions to the Euler--Lagrange equation
\begin{equation}\label{eq:EL}
0 =\mathrm{EL}(L) = \frac{\p L}{\p u} - \sum_{j=0}^{n} \frac{\d }{\d x_j} \frac{\p L}{\p u_{x_j}},
\end{equation}
which constitutes a partial differential equation of 2nd order, unless $L$ is degenerate.

\begin{Example}[Linear motions]\label{ex:LinearMotions}
Denote $x_0=t$.
For any $\theta \colon \R^d \to \R$ with nowhere degenerate Hesse matrix $\mathrm{Hess} \theta$, the field theory described by the Lagrangian density $L(u,u_t) = \theta(u_t)$
yield the Euler--Lagrange equations \[0=\mathrm{EL}(L)=\mathrm{Hess} \theta(u_t)u_{tt}\]
that is solved by affine linear motions $u_{tt}=0$ in $\R^d$.
\end{Example}

\Cref{ex:LinearMotions} illustrates that Lagrangian densities are not uniquely determined by a system's motions.
The level of ambiguity is higher than the ambiguity known as {\em gauge freedom}, which refers to transformations of the Lagrangians that leave the term $\mathrm{EL}(L)$ invariant (\cref{rem:GaugeFreedom} - see \cref{sec:Remarks}). This observation will become important later when we will learn field theories from data. Ambiguity of Lagrangians and the role of symmetry have been discussed in theoretical physics\cite{HENNEAUX198245,Marmo1987,MARMO1989389}.

We proceed with introducing our main examples of field theories.


\begin{Example}[Wave equation]\label{ex:WaveEQ}
For a potential $V \colon \R \to \R$ the Euler--Lagrange equations to the Lagrangian
\begin{equation}\label{eq:WaveL}
L(u,u_t,u_x) = \frac{1}{2}(u_t^2-u_x^2) - V(u)
\end{equation}
yield the non-linear wave equation
\begin{equation}\label{eq:WaveEQ}
u_{tt}(t,x) - u_{xx}(t,x) + \nabla V(u(t,x)) =0.
\end{equation}
Here $\nabla V$ denotes the gradient of $V$ and $u_{tt} = \frac{\p^2 u}{\p t^2}$, $u_{xx} = \frac{\p^2 u}{\p x^2}$ partial derivatives.
\end{Example}

\begin{Example}[Degenerate Lagrangian]\label{ex:DegenerateL}
Let $J=\begin{pmatrix}0&-I_n\\I_n &0\end{pmatrix}$, where $I_n \in \R^{n \times n}$ is an identity matrix. For a function $H \colon \R^{2n} \to \R$ and the Lagrangian
\begin{equation}\label{eq:DegenerateL}
L(u,u_t) = \frac 12 (J^{-1}u)^\top u_t - H(u)
\end{equation}
the Euler--Lagrange equations are
\begin{equation}\label{eq:HamEQ}
J u_t = \nabla H(u).
\end{equation}
Equation \eqref{eq:HamEQ} are Hamilton's equations on $\R^{2n}$ to the Hamiltonian $H$, with symplectic structure on $\R^{2n}$ defined by $J$.

The Lagrangian \eqref{eq:DegenerateL} is degenerate, since the matrix $\frac{\p^2 L}{\p u_t^2} = 0$ is not regular. Indeed, the Euler--Lagrange equations \eqref{eq:HamEQ} are first order differential equations rather than of second order.
\end{Example}



Using a natural extension of the described variational principle \eqref{eq:ActionFunctionalS} to complex valued fields, we present a variational principle for Schrödinger equation\cite{DERIGLAZOV2009}.

\begin{Example}[Schrödinger equation]\label{ex:SE}
For a potential $V \colon \R \to \R$ the Euler--Lagrange equations to the Lagrangian
\begin{equation}\label{eq:SELagrangian}
	L(\Psi,\Psi_t,\Psi_x) = -\frac{\hbar}{2}(\overline{\Psi}\Psi_t - \Psi \overline{\Psi_t}) - \overline{\Psi_x} \Psi_x - V(\overline{\Psi} \Psi)
\end{equation}
yield the non-linear Schrödinger equation
\begin{equation}\label{eq:SE}
	 \hbar \mathrm{i} \Psi_t = (- \Delta + V'(|\Psi|^2) )\Psi.
\end{equation}
Here $\Psi$ is a complex valued function defined on a space-time domain parametrised by the variables $(t,x)$, $\overline{\Psi}$ its complex conjugation, $\hbar$ is the Plank constant, $\mathrm{i}$ the imaginary unit, and $\Delta = \frac{\p^2}{\p x^2}$ denotes the spatial Laplacian.

Alternatively, with $\Psi = \phi + \mathrm{i} p$, $u=\begin{pmatrix}
\phi, & p\end{pmatrix}^\top$, $J$ as in \cref{ex:DegenerateL}, the Schrödinger equation can be obtained from the a real variational principle with Lagrangian
\begin{equation}\label{eq:SERealL}
L(u,u_t,u_x) = \hbar (J^{-1} u)^\top u_t - \| u_x\|^2 - V(\|u\|^2)
\end{equation}
which yields
\begin{equation}\label{eq:SEReal}
\hbar J u_t = (-\Delta + V'(\|u\|^2)) u,
\end{equation}
which is equivalent to \eqref{eq:SE}.

\end{Example}

\subsection{Discrete variational principles} \label{sec:DiscreteVarPrinciples}


\subsubsection{Set-Up}

Let us introduce discrete actions based on various types of discrete Lagrangian densities. These can be interpreted as approximations to exact discrete Lagrangians (see \cref{sec:ExactLd}).
While we provide a rather general setup and expression of discrete Lagrangian densities in the following, \cref{ex:3ptLd} introduces the types of densities used in our numerical experiments.
For a detailed introduction we refer to the article\cite{Marsden2001} by Marsden, Pekarsky, Shkoller, West.

First we introduce some notation for meshes.
Consider
\begin{itemize}
\item the $n+1$-dimensional cube $X=[x^0_0,x^{N_0}_0] \times \ldots \times [x^0_n,x^{N_n}_n] \subset \R^{n+1}$,
\item 
the mesh $X_d = \{x^0_0,\ldots,x^{N_0}_0\} \times \ldots \times \{x^0_n,\ldots,x^{N_n}_n\} \subset X$,
\item the grid's interior points $\mathring X_d = \{x^1_0,\ldots,x^{N_0-1}_0\} \times \ldots \times \{x^1_n,\ldots,x^{N_n-1}_n\} \subset X_d$.
\end{itemize}
Assume for simplicity that the grid $X_d$ is uniform with spacing $\Delta x_r = (x^{N_r}_r-x^0_r)/N_r$ ($r=0,\ldots,n$). Let $\Delta x = \Delta x_0 \cdot \ldots \cdot \Delta x_n$ denote the discrete volume element.
Consider 
\begin{itemize}
\item the index set $\mathcal{I} = \{0,\ldots,N_0\} \times \ldots \times \{0,\ldots,N_n\}$,

\item the multi-index $l = (l_0,\ldots,l_n) \in \mathcal I$,

 
\item the small cube $X^l=[x_0^{l_0},x_0^{l_0+1}] \times \ldots \times [x_n^{l_n},x_n^{l_n+1}]$ to $l \in {\mathcal I}$ with $l_r < N_r$ for all $0\le r \le n$,

\item and the set of vertices $X_\vrtx^l = \{x_0^{l_0},x_0^{l_0+1}\} \times \ldots \times \{x_n^{l_n},x_n^{l_n+1}\}$ of a small cube $X^l$. Each $X_\vrtx^l$ has $2^{n+1}$ elements.
\end{itemize}

A discrete action functional $S_d$ now assigns real values to functions $U \colon X_d \to \R$ defined on the grid $X_d$: a discrete action functional is given as $S_d \colon (\R^d)^{X_d} \to \R$ with
\begin{equation}\label{eq:DiscreteS1}
S_d(U) = \sum_{l \in \mathcal I, l_r < N_r} L_d(X^l_\vrtx,U(X^l_\vrtx)) \Delta x,
\end{equation}
where $L_d \colon (\R^{n+1})^{2^{n+1}} \times (\R^d)^{2^{n+1}} \to \R$ can be evaluated from the values of $U$ on each vertex set $X^l_\mathrm{vrtx}$. $L_d$ is called discrete Lagrangian density.
In many cases $L_d$ is {\em autonomous}, i.e.\ does not explicitly depend on $X^l_\mathrm{vrtx}$ and is, thus, of the form $L_d(U(X^l_\vrtx))$.
The discrete action $S_d$ of \eqref{eq:DiscreteS1} is a discretised version of \eqref{eq:ActionFunctionalS}.

The discrete Euler--Lagrange equations are obtained by discrete variations fixing the boundary as $\nabla_{U(\mathring X_d)} S_d(U)=0$. This gives for any interior vortex $v \in \mathring{X_d}$ the discrete Euler--Lagrange equation
\begin{equation}\label{eq:DEL1}
\mathrm{DEL}(L_d)_v := \nabla_{U(v)} \sum_{l \in \mathcal{I}_v} L_d(X^l_\vrtx,U(X^l_\vrtx)) =0,
\end{equation}
where $\mathcal{I}_v$ contains all indices $l$ of cubes that contain $v$, i.e.\ $v \in X^l$. The derivative $\nabla_{U(v)}$ describes the derivative with respect to the $u$-variable at the vortex $v$.

\subsubsection{Examples}

\begin{Example}[Linear motions]\label{ex:DiscreteLinearMotions}
Denote $x_0=t$. Let $A \in \R^{d \times d}$ be a symmetric, non-degenerate matrix.
Consider a uniform mesh $X_d=\{t^0,\ldots,t^N\}$ on $[t^0,t^N] \subset \R$. We identify $(\R^d)^{X_d}$ with the space $(\R^d)^{N+1}$.
In analogy to \cref{ex:LinearMotions}, for any $\theta\colon \R^d \to \R$ with nowhere degenerate Hessian 
the discrete Lagrangian
\begin{equation*}
L_d(u^i,u^{i+1}) = \theta(u^{i+1}-u^{i}) 
\end{equation*}
yields discrete Euler--Lagrange equations
\[
0=\mathrm{DEL}(L_d)_{u^i} = \nabla \theta (u^{i} - u^{i-1}) -\nabla \theta (u^{i+1} - u^i)
\]
to which $u^{i+1}=u^i + (u^{i}-u^{i-1})$
is the locally unique solution.
\end{Example}

\Cref{ex:DiscreteLinearMotions} demonstrates that discrete Lagrangian densities are not uniquely determined by a system's motion. 
Depending on the choice of $L_d$, the discrete Euler--Lagrange equations can be easier or more difficult to solve numerically. This will need to be considered when learning a model of $L_d$ from data.
The following example introduces the two main stencils used in our numerical experiments.

\begin{figure}
	\includegraphics[width=0.4\linewidth]{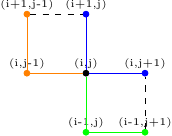}
	\includegraphics[width=0.4\linewidth]{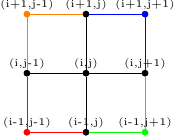}
	\caption{Visualisation of the 7 and 9 point stencil of \cref{ex:3ptLd}. Colours indicate that points are part of the same summand in \eqref{eq:DEL3pt} or \eqref{eq:DEL4pt}. Black vertices are present in several summands.}\label{fig:Stencils}
\end{figure}
\begin{Example}[3 and 4 point discrete Lagrangian]\label{ex:3ptLd}
We denote $x_0 =t$, $x_1 = x$ and identify $\R^{X_d} \cong \R^{(N_0+1) \times (N_1+1)}$ by the space of $(N_0+1) \times (N_1+1)$ matrices with elements $U = (u^i_j)^{0 \le i\le N_0}_{0 \le j\le N_1}$. For a discrete Lagrangian of the form $L_d(u^i_j,u^{i+1}_j,u^i_{j+1})$ (3 input arguments, autonomous) the discrete Euler--Lagrange equations \eqref{eq:DEL1} yield the 7 point stencil
\begin{equation}\label{eq:DEL3pt}
\begin{split}
\mathrm{DEL}(L_d)_j^i
=\frac{\p}{\p u^i_j}
\big(
&\phantom{+}L_d(u^{i}_{j},u^{i+1}_{j},u^{i}_{j+1})\\
&+ L_d(u^{i-1}_{j},u^{i}_{j},u^{i-1}_{j+1})\\
&+ L_d(u^{i}_{j-1},u^{i+1}_{j-1},u^{i}_{j})
\big)
=0,
\end{split}
\end{equation}
denoting $\mathrm{DEL}(L_d)_j^i = \mathrm{DEL}(L_d)_{u^i_j}$.
If $L_d$ is of the form $L_d(u^i_j,u^{i+1}_j,u^i_{j+1},u^{i+1}_{j+1})$ (4 input arguments, autonomous), then the discrete Euler--Lagrange equations yield the 9 point stencil
\begin{equation}\label{eq:DEL4pt}
	\begin{split}
\mathrm{DEL}(L_d)_j^i=
\frac{\p}{\p u^i_j}
		\big(
		&\phantom{+}L_d(u^{i}_{j},u^{i+1}_{j},u^{i}_{j+1},u^{i+1}_{j+1})\\
		&+ L_d(u^{i-1}_{j},u^{i}_{j},u^{i-1}_{j+1},u^{i}_{j+1})\\
		&+ L_d(u^{i}_{j-1},u^{i+1}_{j-1},u^{i}_{j},u^{i+1}_{j})\\
		&+ L_d(u^{i-1}_{j-1},u^{i}_{j-1},u^{i-1}_{j},u^{i}_{j})
		\big)
		=0.
	\end{split}
\end{equation}
The stencils are visualised in \cref{fig:Stencils}
\end{Example}

We proceed by introducing the two main discrete field theories used in the numerical experiments.

\begin{Example}[Discrete 2d wave equation]\label{ex:DiscreteWave}
We denote $x_0 =t$, $x_1 = x$ and identify $\R^{X_d} \cong \R^{(N_0+1) \times (N_1+1)}$ by the space of $(N_0+1) \times (N_1+1)$ matrices with elements $U = (u^i_j)^{0 \le i\le N_0}_{0 \le j\le N_1}$.
Let $L$ be the Lagrangian of \cref{ex:WaveEQ}. Consider the autonomous discrete Lagrangian density
\begin{equation}\label{eq:LdWave2d}
\begin{split}
L_d(u^i_j,u^{i+1}_j,u^{i}_{j+1}) 
= L\left(u^i_j,\frac {u^{i+1}_j - u^i_j}{\Delta t} ,\frac {u^{i}_{j+1} - u^i_j}{\Delta x} \right)\\
= \frac 12 \left( \frac {u^{i+1}_j - u^i_j}{\Delta t} \right)^2 - \frac 12 \left( \frac {u^{i}_{j+1} - u^i_j}{\Delta x} \right)^2 - V(u^i_j).
\end{split}
\end{equation}
The discrete Euler--Lagrange equation \eqref{eq:DEL3pt} is given as
\begin{equation}\label{eq:DELWave2d}
\frac{u^{i-1}_{j}-2u^{i}_{j}+u^{i+1}_{j}}{\Delta t^2}
-\frac{u^{i}_{j-1}-2u^{i}_{j}+u^{i}_{j+1}}{\Delta x^2}
+ \nabla V(u^{i}_{j})=0.
\end{equation}
It can be interpreted as a 2nd order discretisation of \cref{ex:WaveEQ} by finite differences.
While for general $L_d$ the discrete Euler--Lagrange equation \eqref{eq:DEL3pt} constitutes a 7-point stencil, on $L_d$ of \eqref{eq:LdWave2d} it simplifies to a 5-point stencil.
The stencil can be used conveniently to compute forward propagations of solutions as $u^{i+1}_j$ can be computed in terms of $u^i_j$, $u^{i-1}_j$, $u^i_{j+1}$, which all belong to previous time-steps.
\end{Example}


The following example uses a straight forward extension of the discrete variational principle to complex valued fields.

\begin{Example}[Discrete Schrödinger equation]\label{ex:DiscreteSE}
Let $L$ be the Lagrangian \eqref{eq:SELagrangian} of \cref{ex:SE}. Consider
\begin{equation}\label{eq:SELd}
L_d(\Psi^{i-1}_{j-1},\Psi^{i}_{j-1},\Psi^{i-1}_{j},\Psi^{i}_{j})
= L (m^i_j\Psi,\Delta t^i_j\Psi, \Delta x^i_j\Psi)
\end{equation}
with 
\begin{align}
m^i_j\Psi &= \frac{\Psi^{i-1}_{j-1}+\Psi^{i-1}_{j}+\Psi^{i}_{j}+\Psi^{i}_{j-1}}{4}\\
\Delta t^i_j\Psi &= \frac{(\Psi^{i}_{j-1}-\Psi^{i-1}_{j-1})+(\Psi^{i}_{j}-\Psi^{i-1}_{j})}{2 \Delta t}\\
\Delta x^i_j\Psi &= \frac{(\Psi^{i-1}_{j}-\Psi^{i-1}_{j-1})+(\Psi^{i}_{j}-\Psi^{i}_{j-1})}{2 \Delta x}
\end{align}
that approximate the field $\Psi$ and its derivative in $t$ and $x$ in the centre of a square $X^{(i,j)}$ of the gird. Further, consider the following approximations of derivatives at the node $(i,j)$
\begin{align*}
{\D_t}^i_j\Psi &= \frac{(\Psi^{i+1}_{j-1}-\Psi^{i-1}_{j-1})+2(\Psi^{i+1}_{j}-\Psi^{i-1}_{j})+(\Psi^{i+1}_{j+1}-\Psi^{i-1}_{j+1})}{8 \Delta t}\\
{\D_x^2}^i_j\Psi &= \frac 1 {4 \Delta x^2} \Big( (\Psi^{i-1}_{j-1}-2\Psi^{i-1}_{j} + \Psi^{i-1}_{j+1} )\\
&\phantom{= \frac 1 {4 \Delta x^2} \Big(}
+(\Psi^{i}_{j-1}-2\Psi^{i}_{j} + \Psi^{i}_{j+1} ) \\
&\phantom{= \frac 1 {4 \Delta x^2} \Big(}+ (\Psi^{i+1}_{j-1}-2\Psi^{i+1}_{j} + \Psi^{i+1}_{j+1} ) \Big).
\end{align*}
As discrete Euler--Lagrange equations \eqref{eq:DEL4pt} we obtain
\begin{equation}\label{eq:SEDEL}
\begin{split}
\mathrm i \hbar {\D_t}^i_j\Psi
&= - {\D_x^2}^i_j\Psi\\
&+ \frac 14 \Big(
V'(m^{i-1}_{j-1}\Psi)m^{i-1}_{j-1}\Psi
+ V'(m^{i}_{j-1}\Psi)m^{i}_{j-1}\Psi\\
&+\phantom{\frac 14 \Big(} V'(m^{i-1}_{j}\Psi)m^{i-1}_{j}\Psi
+ V'(m^{i}_{j}\Psi)m^{i}_{j}\Psi\Big),
\end{split}
\end{equation}
which is a 9 point stencil.
\end{Example}

\subsubsection{Forward propagation}\label{rem:comp34Ld}


In \cref{ex:3ptLd} assume that boundary conditions $(u_0^i)^{0\le i \le N_0}$, $(u_{N_1}^i)^{0 \le i \le N_0}$, initial data $(u_j^0)_{0 < j < N_1}$, and initial velocities $(\dot u_j^0)_{0 < j < N_1}$ are given.
	Let 
	\begin{equation}\label{eq:LdelXdelT}
		L_{\Delta x }^{\Delta t} (U^i,U^{i+1}) 
		= \sum_{j=0}^{N_1-1} L_d(u^i_j,u_j^{i+1},u^i_{j+1})
	\end{equation}
	or
	\begin{equation}\label{eq:LdelXdelT4}
		L_{\Delta x }^{\Delta t} (U^i,U^{i+1}) = \sum_{j=0}^{N_1-1} L_d(u^i_j,u_j^{i+1},u^i_{j+1},u^{i+1}_{j+1}),
	\end{equation}
	respectively. Here $U^i = (u^i_j)_{0<j<N_1}$ refers to values at interior spatial grid points. Let $L_{\Delta x}(U,\dot U) = L_{\Delta x }^{\Delta t} (U-\Delta t/2 \dot U,U+\Delta t/2 \dot U)$.
	Solutions to the discrete field theory can be computed as follows:
	$U^1$ is computed as the solution to
	\begin{equation}\label{eq:initVelocityProfile}
		\nabla_{\dot U} L_{\Delta x}(U^0,\dot U^0) = - \nabla_{U^1} L_{\Delta x }^{\Delta t} (U^0,U^{1}).
	\end{equation}
	Subsequent $U^i$ are computed as solutions to the DEL for $L_{\Delta x }^{\Delta t}$
	\begin{equation}\label{eq:DELLdxU}
		\nabla_{U^{i}} \left(
		L_{\Delta x}^{\Delta t}(U^{i-1},U^{i})
		+L_{\Delta x}^{\Delta t}(U^{i},U^{i+1})
		\right) =0.
	\end{equation}
	This follows classical variational integration theory\cite{MarsdenWestVariationalIntegrators}.
	We refer to the literature on low-rank approximations and integration schemes\cite{Kieri2016,Lubich2018,WalachThesis,Schrammer2022Thesis} in case fine meshes or higher spacial dimensions are considered and the systems \eqref{eq:initVelocityProfile}, \eqref{eq:DELLdxU} become prohibitively high-dimensional.
	
	In the special case that $L_{\Delta x}(U,\dot U)$ is linear in the velocities $\dot U$, \eqref{eq:initVelocityProfile} is independent of $\dot U^0$ such that no initial velocity profile $\dot U^0$ is needed to obtain $U^1$ via \eqref{eq:initVelocityProfile}. In particular, if $L_{\Delta x}(U,\dot U)$ is of the form \eqref{eq:DegenerateL}, then the computation is numerically stable\cite{Marsden2002DegenerateLagrangians}. Indeed, independence of $\dot U^0$ of the numerical scheme is expected since the EL equations \eqref{eq:HamEQ} are 1st order ordinary differential equations.
	While numerical stability of the scheme is guaranteed for $L_{\Delta x}$ of the form \eqref{eq:DegenerateL}, for more general Lagrangians linear in velocities other techniques are required for numerically stable computations\cite{kraus2017projected,Kraus2018,Burby2022}.

\begin{Remark}\label{rem:useLocalStencil}
Depending on boundary conditions, the solution procedure of \cref{rem:comp34Ld} can often be drastically simplified by using the local stencils \eqref{eq:DEL3pt}, \eqref{eq:DEL4pt} to propagate solutions forward in time, as we made explicit in \cref{ex:DiscreteWave}. Use of the local stencils avoids solving the potentially high dimensional system \eqref{eq:DELLdxU}.
\end{Remark}

\section{Preparation of regularisation strategy}\label{sec:RegPrepare}

In the following, we will train a neural network model of a discrete Lagrangian density $L_d$ such that $L_d$ defines a discrete field theory that is consistent with observations.
Discrete Lagrangian densities are not uniquely determined by the system's motions as illustrated in \cref{ex:DiscreteLinearMotions}. (Also see \cref{rem:gaugeLd34}.)
The secondary goal is to identify a model of $L_d$ that is not only consistent with the true dynamics but is also computationally efficient when numerical methods are used to solve the discrete Euler--Lagrange equations.

To prepare our design of regularisers to be employed in the training phase, we analyse the convergence properties of Newton interations when used to compute solutions to discrete Euler--Lagrange equations for the 3 and 4 point discrete Lagrangians of \cref{ex:3ptLd}.
The following \namecref{prop:SolveDELNewton} refers to the DEL for 3-point Lagrangians \eqref{eq:DEL3pt}.

\begin{Proposition}\label{prop:SolveDELNewton}
	Let $u^{i}_{j}$, $u^{i+1}_{j}$, $u^{i}_{j+1}$, $u^{i-1}_{j}$, $u^{i-1}_{j+1}$, $u^{i}_{j-1}$, $u^{i+1}_{j-1}$ such that \eqref{eq:DEL3pt} holds. Let $\mathcal O\subset \R^d$ be a convex, neighbourhood of $u^\ast=u^{i+1}_{j}$, $\| \cdot \|$ a norm of $\R^d$ inducing an operator norm on $\R^{d \times d}$.
	Define $p(u) :=\frac{\p^2 L_d}{\p u^{i}_{j}\p u}(u^i_j,u,u^i_{j+1})$ and let $\theta$ and $\overline{\theta}$ be Lipschitz constants on $\mathcal O$ for $p$
	and for $\mathrm{inv} \circ p$, respectively, where $\mathrm{inv}$ denotes matrix inversion.
	Let
	\begin{equation}\label{eq:rho}
		\rho^\ast := \left\|\mathrm{inv}(p(u^\ast))\right\|
		= \left\|\left(\frac{\p^2 L_d}{\p u^{i}_{j}\p u^\ast}(u^i_j,u^\ast,u^i_{j+1})\right)^{-1}\right\|
	\end{equation}
	and let $f(u^{(n)})$ denote the left hand side of \eqref{eq:DEL3pt} with $u_j^{i+1}$ replaced by $u^{(n)}$.
	If $\|{u}^{(0)}-{u}^\ast\| \le \min\left(\frac{\rho^\ast}{\overline{\theta}}, \frac{1}{2\theta\rho^\ast}\right)$ for ${u}^{(0)}\in \mathcal{O}$, then the Newton iterations
	${u}^{(n+1)}:= {u}^{(n)} - \mathrm{inv}(p(u^{(n)})) f(u^{(n)})$
	converge quadratically against ${u}^{\ast}$, i.e.\
	\begin{equation}\label{eq:NewtonQuadraticConvergence}
		\|{u}^{(n+1)} - {u}^{\ast} \| \le \rho^\ast \theta \|{u}^{(n)} - {u}^{\ast} \|^2. 
	\end{equation}
	
\end{Proposition}

A proof can be found in \cref{sec:ProofConvergence7Stencil}.

\begin{Remark}
\cref{prop:SolveDELNewton} is a restatement of Proposition 1 of our conference contribution\cite{DLNNDensity}, where it is stated without proof.
\end{Remark}

\Cref{prop:SolveDELNewton} clarifies numerical convergence properties when computing solutions on a grid with the stencil visualised to the left of \cref{fig:Stencils}. It applies when values $u_0^i$ at the left boundary and initial values $u_j^1$, $u_j^2$ are known for all indices $i,j$ such that remaining values can be computed subsequently.
However, sometimes it is required to compute all spatial data at a time step at once due to boundary conditions, particularities of the considered stencils or numerical stability issues. This is the setting of the computations mentioned in \cref{rem:comp34Ld}.
To obtain a field theory optimised for such computations, we present the following \namecref{prop:SolveDELNewtonU}.

\begin{Proposition}\label{prop:SolveDELNewtonU}
	Let $U^{i-1}$, $U^{i}$, $U^{i+1} \in (\R^d)^{(N_1-1)}$ such that \eqref{eq:DELLdxU} holds.
	Let $\mathcal O\subset (\R^d)^{(N_1-1)}$ be a convex neighbourhood of $u^\ast=U^{i+1}$, $\| \cdot \|$ a norm on $(\R^d)^{(N_1-1)}$ inducing an operator norm on $(\R^d)^{(N_1-1) \times (N_1-1)}$.
	Define $p(U) :=\frac{\p^2 L_{\Delta x}^{\Delta t}}{\p U^{i}\p U}(U^i,U)$ and let $\theta$ and $\overline{\theta}$ be Lipschitz constants on $\mathcal O$ for $p$
	and for $\mathrm{inv} \circ p$, respectively, where $\mathrm{inv}$ denotes matrix inversion.
	Let
	\begin{equation}\label{eq:rhoU}
		\rho^\ast := \left\|\mathrm{inv}(p(U^\ast))\right\|
		= \left\|\left(\frac{\p^2 L_{\Delta x}^{\Delta t}}{\p U^{i}\p U^\ast}(U^i,U^\ast)\right)^{-1}\right\|
	\end{equation}
	and let $f(U^{(n)})$ denote the left hand side of \eqref{eq:DELLdxU} with $U^{i+1}$ replaced by $U^{(n)}$.
	If $\|{U}^{(0)}-{U}^\ast\| \le \min\left(\frac{\rho^\ast}{\overline{\theta}}, \frac{1}{2\theta\rho^\ast}\right)$ for ${u}^{(0)}\in \mathcal{O}$, then the Newton iterations
	${U}^{(n+1)}:= {U}^{(n)} - \mathrm{inv}(p(U^{(n)})) f(U^{(n)})$
	converge quadratically against ${U}^{\ast}$, i.e.\
	\begin{equation}\label{eq:NewtonQuadraticConvergenceU}
		\|{U}^{(n+1)} - {U}^{\ast} \| \le \rho^\ast \theta \|{U}^{(n)} - {U}^{\ast} \|^2. 
	\end{equation}
	
\end{Proposition}

\begin{proof}
The proof is analogous to the proof of \cref{prop:SolveDELNewton}, which can be found in \cref{sec:ProofConvergence7Stencil}.
\end{proof}

For computational details of how to approximate $\rho^\ast$ numerically, see \cref{sec:ComputeRhoAst}.

\section{Loss function and regulariser for training discrete Lagrangian densities}\label{sec:LossFuns}

Assume we are given observations $U^{(k)} \colon X_d \to \R^d$ $k=1,\ldots,K$ of solutions to an unknown discrete field theory and the goal is to identify a model of the discrete field theory.
We pick a type of a discrete Lagrangian density $L_d$ such as the 3 or 4 points discrete Lagrangians introduced in \cref{ex:3ptLd} or any discretisation of \eqref{eq:ExactLd} by a linear quadrature formula. The discrete Lagrangian $L_d$ is modelled as a neural network whose parameters are optimised such that
\begin{enumerate}
	
\item the discrete Euler--Lagrange equations \eqref{eq:DEL1} are consistent with the observations,

\item computations based on \eqref{eq:DEL1} are numerically well conditioned.
	
\end{enumerate}

For this, we consider a loss function $\ell = \ell_{\mathrm{data}} + w\ell_{\mathrm{reg}}$ consisting of the data consistency term $\ell_{\mathrm{data}}$ and the regulariser $\ell_{\mathrm{reg}}$ and a weight $w \in \R_+$. These are specified in the following.

\subsection{Data consistency}

We introduce the data consistency term
\begin{equation}\label{eq:elldata}
\ell_{\mathrm{data}} = \sum_{k=1}^{K}\sum_{v} |\mathrm{DEL}(L_d)_v(U^{(k)})|^2.
\end{equation}
In the inner sum, the sum is taken over all vertices for which the summand is defined.
The term $\ell_{\mathrm{data}}$ measures how consistent the model $L_d$ is with the training data. 

\begin{Remark}[Partial observations and batch learning]
The computation of a summand $|\mathrm{DEL}(L_d)_v(U^{(k)})|^2$ in \eqref{eq:elldata} does not involve the full observation $U^{(k)}(X_d)$ but only the observation on a stencil such as the ones illustrated in \cref{fig:Stencils}. 
Therefore, partial observations can be incorporated in $\ell_{\mathrm{data}}$. Moreover, when batch learning\cite{MarslandMachineLearning} is applied, batches can contain tuples of stencil data coming from different observed solutions. This can be expected to prevent over-fitting phenomena.
\end{Remark}

\begin{Remark}[Low dimensionality]
Although solutions to the discrete field theory can require solving high dimensional systems such as $|\mathrm{DEL}(L_{\Delta x}^{\Delta t})^i(U)|^2=0$ in \eqref{eq:DELLdxU}, in the training phase to compute $\ell_{\mathrm{data}}$ each summand only requires evaluating the low-dimensional Euler--Lagrange equations \eqref{eq:DEL1}.
\end{Remark}

\begin{Remark}[Locality]
In the setting of the article, we assume that there exists a continuous field theory underlying our observational data $U^{(k)}$ on a mesh $X_d$ (spacetime lattice). The underlying continuous theory can be described by a partial differential equation \eqref{eq:EL} and whether a function $u \colon X \to \R$ fulfils the differential equation can be checked point-wise, i.e.\ it is a local condition. This locality is reflected in our discrete model by the locality of the stencil defining the term $\mathrm{DEL}(L_d)_v(U^{(k)})$. In this sense, the data-driven model will be local in space and time by construction.
In contrast, if in the setting of a space-time lattice $X=[t^0,t^N] \times [x^0,x^N]$ we were to learn $L_{\Delta x}^{\Delta t}$ (see \eqref{eq:LdelXdelT4}) rather than the discrete Lagrangian density $L_d$, potentially after model order reduction in the spatial directions, then locality would only be guaranteed in the time-direction.
\end{Remark}

\subsection{Regulariser}

Since any constant function $L_d$ constitutes a minimiser of $\ell_{\mathrm{data}}$, a carefully chosen regulariser $\ell_{\mathrm{reg}}$ is required to obtain a regular discrete field theory.
Indeed, as seen in \cref{ex:DiscreteLinearMotions} much freedom in the choice of $L_d$ can be expected.
In view of \namecref{prop:SolveDELNewton} \labelcref{prop:SolveDELNewton,prop:SolveDELNewtonU}, we consider regularisers that aim to minimise $\rho^\ast$ of \eqref{eq:NewtonQuadraticConvergence} or \eqref{eq:NewtonQuadraticConvergenceU}, respectively. If successfully optimised, we can expect solvability and rapid convergence of Newton iterations when computing solutions to the field theory, at least for initial data close to training data.

A regulariser $\ell_{\mathrm{reg}}$ for the 3-point discrete Lagrangian of \cref{ex:3ptLd} is given by
\begin{equation}\label{eq:regLd3pt}
\ell_{\mathrm{reg}}=\alpha \sum_{k=1}^{K}\sum_{i,j} \left\|\left(\frac{\p^2 L_d}{\p u^{i}_{j}\p u^{i+1}_j}({u^i_j}^{(k)},{u^{i+1}_j}^{(k)},{u^i_{j+1}}^{(k)})\right)^{-1}\right\|^2
\end{equation}
with normalisation $\alpha =  \frac{1}{KN_0 N_1}$ by the number of summands. In the summation in \eqref{eq:regLd3pt} the indices $i,j$ run over all interior mesh points.
The expression is motivated by \cref{prop:SolveDELNewton}:
each summand corresponds to $(\rho^\ast)^2$ of the \namecref{prop:SolveDELNewton} which is related to the convergence speed of Newton iterations. We obtain the bound $\rho^\ast < \sqrt{\ell_{\mathrm{reg}}}$.
The resulting field theory will be optimised for propagating solutions forward in space and time using the 7-point stencil \cref{eq:DEL3pt} visualised to the left of \cref{fig:Stencils}.

To obtain a field theory based on the 3- or 4-point discrete Lagrangian of \cref{ex:3ptLd} which is optimised for forward propagation of solutions as described in \cref{rem:comp34Ld}, we consider the regularisation
\begin{equation}\label{eq:regLdxdt}
\ell_{\mathrm{reg}} 
=  \frac{1}{K N_0} \sum_{k=1}^{K}\sum_{i=1}^{N_0}  \left\|{\Lambda_i^{(k)}}^{-1}\right\|^2,
\end{equation}
with
\begin{equation}
\Lambda^{(k)}_i =
\frac{\p^2 L_{\Delta x}^{\Delta t}}{\p U^{i}\p U^{i+1}}({U^i}^{(k)},{U^{i+1}}^{(k)})
\end{equation}
The expression is based on \cref{prop:SolveDELNewtonU}. Here ${U^i}^{(k)} = ({u^i_j}^{(k)})_{0< j < N_1}$.
See \cref{rem:periodicBD} and above for a description of $\Lambda^{(k)}$.
In practise, we found that the tamed version
\begin{equation}\label{eq:regLdxdtTamed}
\ell_{\mathrm{reg}}
=  \frac{1}{K N_0}\sum_{k=1}^{K}\sum_{i=1}^{N_0} \mathrm{relu}\left( -10 \left\|{\Lambda^{(k)}}^{-1}\right\|^{-2} +1\right)\le 1
\end{equation}
is more suitable for training. Here $\mathrm{relu}(x) = \max(0,x)$ is the rectified linear unit function. For numerical computation of  $\left\|{\Lambda^{(k)}}^{-1}\right\|^{-1}$ see \cref{rem:ApproxRho}.
The regularisers \eqref{eq:regLdxdt} and \eqref{eq:regLdxdtTamed} both get large when the singular values of $\Lambda^{(k)}$ become small on the training data. However, \eqref{eq:regLdxdtTamed} is tamed in the sense that the summands take values in $[0,1)$ and are zero when the smallest singular values of $\Lambda^{(k)}$ is bigger than or equal to $0.1$. On the training data set, $\rho^\ast$ of \cref{prop:SolveDELNewtonU} is bounded by $\rho^\ast \le \sqrt{10(1-\ell_{\mathrm{reg}})^{-1}}$.
 \Cref{fig:RegulariserTaming} shows the behaviour of the summands of various regularisers in terms of the square of the smallest singular value $\sigma^2$.

\begin{figure}
	\includegraphics[width=0.32\linewidth]{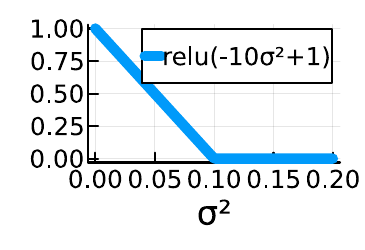}
	\includegraphics[width=0.32\linewidth]{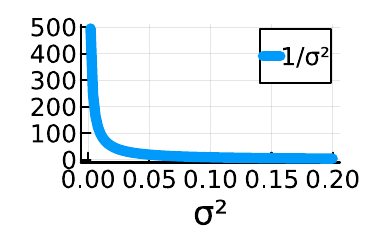}
	\includegraphics[width=0.32\linewidth]{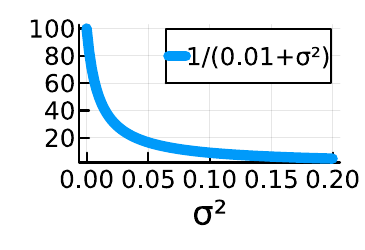}
	\caption{Behaviour of summands in regularisers $\ell_{\mathrm{reg}}$ of \eqref{eq:regLd3pt}, \eqref{eq:regLdxdt}, \eqref{eq:regLdxdtTamed}, and $\ell_{\mathrm{reg}}^{\mathrm{illustration}}$. }\label{fig:RegulariserTaming}
\end{figure}

\begin{figure}
	\includegraphics[width=0.45\linewidth]{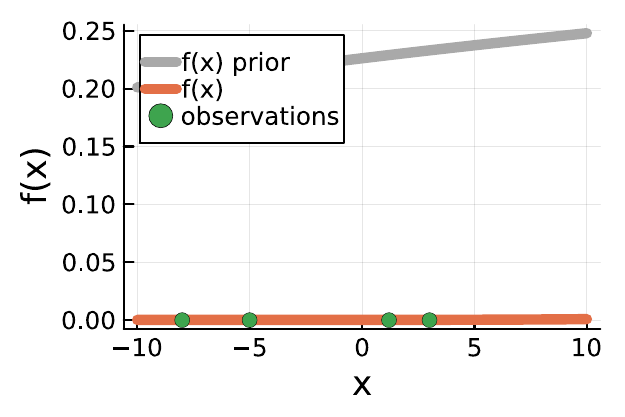}
	\includegraphics[width=0.45\linewidth]{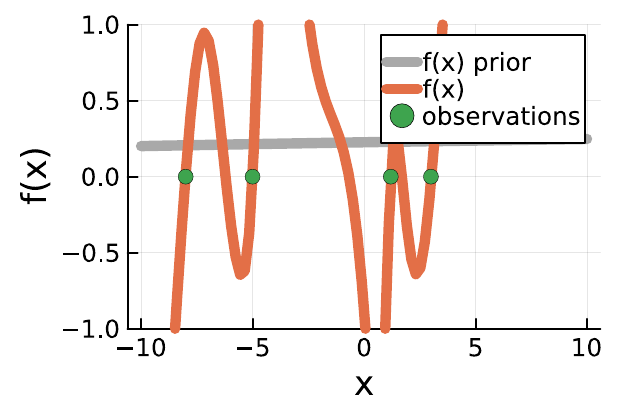}
	\caption{Illustration of degeneracy of roots when learning without (left) or with (right) regularising terms. The orange curve refers to the trained model, the grey curve to the untrained model of $f$.
	The regulariser makes sure that the roots are non-degenerate and numerical root finding is well-conditioned.	
	}\label{fig:RegulariserIllustration}
\end{figure}
\begin{figure}
	\includegraphics[width=0.7\linewidth]{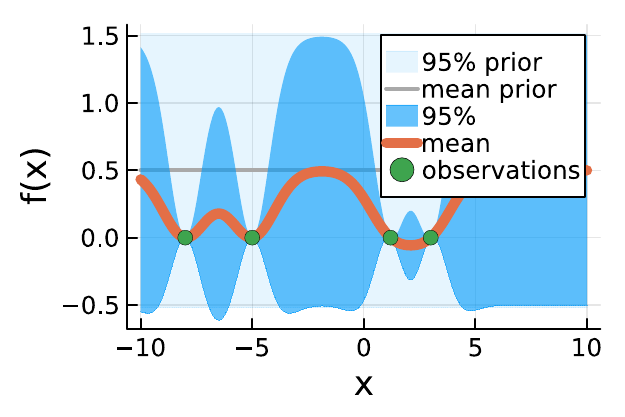}
	\caption{As comparison to \cref{fig:RegulariserIllustration} (left), a posterior distribution for a Gaussian Processes for the same observations is computed without regularisation. The problem of degenerate roots of the mean is less pronounced but does occur for sparse training data.}\label{fig:RegulariserIllustrationGP}
\end{figure}
For a better intuition of the action of the considered regularisers, \cref{fig:RegulariserIllustration} shows the effect of a regulariser in a simplified learning task: a neural network model of a function $f\colon \R \to \R$ is fitted to observations $f(x)=0$ for $x \in R=\{-8,-5,1.2,3\}$. While a model trained with the mean squared error $\ell_{\mathrm{data}}^{\mathrm{illustration}}= \sum_{x \in R} f(x)^2$ is almost flat after $1100$ optimisation steps with {\tt adam}\cite{Adam} (left of \cref{fig:RegulariserIllustration}), the presence of the regulariser $\ell_{\mathrm{reg}}^{\mathrm{illustration}} = \sum_{x \in R}(0.01+f'(x)^2)^{-1}$ drives the model of $f$ to a representation which has nondegenerate roots at $R$ that are easy to find using Newton iterations (right of \cref{fig:RegulariserIllustration}). Indeed, the loss of the regulariser $\ell_{\mathrm{reg}}^{\mathrm{illustration}} < 0.855$ guarantees that $f'(x) > 3.38$ for $x \in R$ and thus the regularity of the roots problem for $f$ close to $R$ is guaranteed.

\begin{Remark}[Comparison with Gaussian Processes]
When Gaussian Processes are used instead of neural networks, the problem of degenerate roots are not as pronounced: \cref{fig:RegulariserIllustrationGP} shows the posterior distribution and its mean of a Gaussian Process for observations $f(x)=0$, $x \in R$ (constant prior $m_0(x)=0.5$, squared exponential kernel with unit length scale and unit standard deviation).
The roots of the mean function at $1.2$ and $3$ are non-degenerate, while roots at $-8$ and $-5$ appear tangential.
This behaviour is quite distinct from that of neural networks \cref{fig:RegulariserIllustration} (left). We conclude that each learning approaches requires adapted regularisation. 
We have developed regularisation strategies for Lagrangians modelled as Gaussian Processes in previous work\cite{LagrangianShadowIntegrators}.
\end{Remark}

\section{Variational principles for travelling waves in discrete and continuous theories}\label{sec:PSC}

This section discusses the notion of travelling waves in discrete field theories. As reference solutions for our numerical experiments in \cref{sec:Experiments}, we report analytic expressions for travelling waves of the discrete wave equation and discrete Schrödinger equation. 
Further, a central motivation for learning (discrete) field theories by learning a (discrete) Lagrangian density is the presence of variational structure of the learned model. We invoke Palais' principle of symmetric criticality\cite{palais1979}. It implies that variational structure is passed on to the governing equations of motions of travelling waves if the discrete Lagrangian is autonomous. In the broader context of learning field theories and trying to predict highly symmetric solutions with the trained model, this is used as a motivation for learning discrete Lagrangian densities with the correct invariances.
Readers mainly interested in the machine learning architecture may skip to the numerical experiments (\cref{sec:Experiments}) directly.

Highly symmetric solutions $u$ of a continuous field theory, such as travelling waves, are of special interest due to their simple structure.
When the same dynamical system is described by a discrete field theory, however, the values $u(X_d)$ of $u$ over the mesh $X_d$ might not have symmetry properties unless the symmetries of $u$ happen to align with the symmetries of the mesh $X_d$. For instance, let $X=\R \times [0,b]/\sim$. Here $[0,b]/\sim$ denotes the identification of endpoints $0 \sim b$ of an interval $[0,b]$ (periodic boundary conditions).
Functions $u \colon X \to \R$ with $u(t,x)=u(t+s,x+cs)$ for all $s \in \R$ are of the form $u(t,x)=f(x-ct)$ for $f \colon [0,b]/\sim \to \R$ (travelling wave with wave speed $c \in \R$).
Consider a rectangular mesh $X_d \subset X$ with mesh widths $\Delta t$, $\Delta x >0$. If and only if $\frac{c \Delta t}{\Delta x} \in \mathbb{Q}$, the values of the mesh $u(X_d)$ are periodic in time and $u(X_d)$ is a discrete travelling wave, i.e.\ $u^i_j = u^{i+z k_0}_{j+z k_1}$ for all $z \in \Z$, where $\frac{c \Delta t}{\Delta x} = \frac{k_0}{k_1}$ with $k_0 \in \Z$, $k_1 \in \N$.
The theory of lattice differential equations\cite{MalletParet2003,Hupkes2022} provides a way to treat $u(X_d)$ as a travelling wave even if $\frac{c \Delta t}{\Delta x} \not \in \mathbb{Q}$.

To design machine learning architectures for discrete field theories that effectively capture solutions corresponding to travelling waves in the continuous theory, we pay close attention to invariances of the field theory that relate to structural properties of the dynamical system that governs travelling waves.
In particular, autonomy of a field theory relates to the fact that travelling waves themselves are governed by a variational principle. This is made precise in the following \namecref{prop:PalaisTW}.

\begin{Proposition}[Symmetric criticality for travelling waves]\label{prop:PalaisTW}
Let $c \in \R \setminus \{0\}$, $b>0$, $X=([0,b/c]/\sim) \times ([0,b]/\sim)$ be a torus (periodic boundary conditions), $W=C^1(X,\R^d)$ continuously differentiable functions and $\Sigma = \{u \in W \,|\, u(t+s,x+cs)=u(t,x)\}$ be travelling waves with wave speed $c$.
Let $S \colon W \to \R$
\begin{equation}\label{eq:STWPalais}
S(u)=\int_0^{b/c}\int_0^b L(u(t,x),u_t(t,x),u_x(t,x)) \d x \d t
\end{equation}
be a continuously differentiable action functional with an autonomous Lagrangian $L \colon (\R^d)^3 \to \R$. 
Denote the restriction of $S$ to $\Sigma$ by $S_\Sigma \colon \Sigma \to \R$.
Then a travelling wave $u \in \Sigma$ is a stationary point of $S$ if and only if $u$ is a stationary point of $S_\Sigma$. Using the identification $\Sigma \cong C^1([0,b]/\sim,\R^d)$, $u(t,x)=f(x-ct)=f(\xi)$ for $u \in \Sigma$, we have
\begin{equation}\label{eq:SSigmaL}
S_\Sigma(f) = \frac bc \int_0^b L(f(\xi),-c f_\xi(\xi),f_\xi(\xi)) \d \xi.
\end{equation}
\end{Proposition}

\begin{Remark}\label{rem:PSCStandingTW}
The statement of \cref{prop:PalaisTW} holds true for standing travelling waves ($c=0$), when the fraction $b/c$ in the expression of the domain $X$ as well as in \eqref{eq:STWPalais} and in \eqref{eq:SSigmaL} are substituted by 1. To include the case $c=0$ we will omit the irrelevant factor $b/c$ when reporting $S_\Sigma$.
\end{Remark}

Notice that for $u \in \Sigma$ to be a stationary point of $S_\Sigma$, the action needs to be stationary at $u$ with respect to variations through travelling waves. To be a stationary point of $S$, however, $u$ needs to be stationary with respect to all variations $\delta u \colon X \to \R^d$.
The \namecref{prop:PalaisTW} says that stationarity of $S_\Sigma$ is already sufficient to conclude stationarity of $S$.

The proof checks invariance of $S(u)$ under the transformation $u \mapsto s.u$ with $s.u(t,x)=u(t+s,x+cs)$ and verifies the conditions of Palais' principle of symmetric criticality (PSC)\cite{palais1979,Torre2011}. A proof is contained in \cref{sec:ProofPalaisTW}.

In particular, \cref{prop:PalaisTW} says that travelling waves are again governed by a variational principle, namely by the restricted functional $S_\Sigma \colon \Sigma \to \R$. We illustrate this on the wave equation and the Schrödinger equation.

\begin{Example}[PSC for wave equation]\label{ex:PSCWave}
Consider the wave equation (\cref{ex:WaveEQ}) with periodic boundary conditions. Let $\xi = x-ct$. 
By \cref{prop:PalaisTW}, travelling wave solutions $u(t,x)=f(x-ct)$ of \eqref{eq:WaveEQ} exactly correspond to the stationary points of $S_\Sigma(f) \colon [0,b]/\sim \to \R$ with
\begin{equation}\label{eq:SPSCWave}
\begin{split}
S_\Sigma = \int_{0}^b \left(\frac 12(c^2-1)f_\xi^2-V(f) \right) \d \xi.
\end{split}
\end{equation}

Although \cref{prop:PalaisTW} and \cref{rem:PSCStandingTW} guarantee variational structure of travelling wave solutions for any $c \in \R$, they do not imply existence of non-trivial solutions: for $V(f)=\frac 12 f^2$ a Fourier ansatz in the Euler--Lagrange equation $(1-c^2)f_{\xi \xi }=\nabla V(f)$ reveals the condition
\[
c^2 = 1+\frac{l^2}{4 \pi^2 m^2}, \quad m \in \Z \setminus \{0\}
\]
with solutions $u(t,x)= \alpha_1 \sin (\kappa(x-ct))+\alpha_2 \cos(\kappa(x-ct))$,
where $\kappa = 2\pi m/b$, $\alpha_1,\alpha_2 \in \R$.
\end{Example}

\begin{Example}[PSC in Schrödinger equation]
In analogy to \cref{ex:PSCWave}, travelling waves $\Psi(t,x)=f(x-ct) = f(\xi) \in \mathbb{C}$ with wave speed $c \in \R$ in the Schrödinger equation (\cref{ex:SE}) with periodic boundary conditions are governed by the variational principle $S_\Sigma \colon [0,b]/\sim \to \R$ with $S_\Sigma(f) = \int_0^b L_\Sigma (f,f_\xi) \d \xi$ with
\begin{equation}
L_\Sigma (f,f_\xi) = \frac{c \hbar}{2}(\overline{f}f_\xi - f \overline{f_\xi})-f_\xi \overline{f_\xi}-V(f\overline{f}).
\end{equation}
A Fourier ansatz for $f$ reveals that for any $m \in \Z \setminus \{0\}$, $\alpha \in \R$ there is a travelling wave $\Psi(t,x)=f(\xi)=\alpha e^{\mathrm i \frac{2 \pi m}{b} \xi}$ with $\xi = x-ct$ and
\begin{equation}\label{eq:WaveSpeedSE}
c = \frac{2 \pi m}{\hbar b} - \frac{b}{2 \pi \hbar}\nabla V(\alpha^2).
\end{equation}
\end{Example}

Let us now obtain corresponding statements for discrete field theories.
To make sense of travelling waves in discrete equations such as \eqref{eq:DEL3pt} and \eqref{eq:DEL4pt}, standard approaches in lattice differential equations\cite{MalletParet2003,Hupkes2022} replace discrete equations by functional equations. We recover this interpretation from a variational perspective.

\begin{Lemma}[Continuous action for discrete field theory]
Let $X = ([0,b_0]/\sim)  \times ([0,b_1]/\sim) $ be a torus and $L_d$ the 3-point discrete Lagrangian density of \cref{ex:3ptLd}. Let $\Delta t, \Delta x >0$. 
Consider the action $S \colon X \to \R$, 
\begin{equation}\label{eq:SLdCont}
S(u)=\int_X L_d(u(t,x),u(t+\Delta t,x),u(t,x+\Delta x))\d x \d t.
\end{equation}
Stationary points $u \colon X \to \R^d$ are the solutions of the functional equation
\begin{equation}\label{eq:DEL3ptFunc}
	\begin{split}
		\frac{\p}{\p u}
		\big(
		&\phantom{+}L_d(u(t,x),u(t+\Delta t,x),u(t,x+\Delta x))\\
		&+ L_d(u(t-\Delta t,x),u(t,x),u(t-\Delta t,x+\Delta x))\\
		&+ L_d(u(t,x-\Delta x),u(t+\Delta t,x-\Delta x),u(t,x))
		\big)
		=0.
	\end{split}
\end{equation}
\end{Lemma}

Notice that \eqref{eq:DEL3ptFunc} is \eqref{eq:DEL3pt} with $u^{i+s}_{j+r}$ substituted by $u(t+s \Delta t,x+r\Delta x)$. 
The proof follows by a standard variational calculus. Analogous statements hold for all autonomous discrete Lagrangians \eqref{eq:DiscreteS1}, in particular for the 4-point discrete Lagrangian of \cref{ex:3ptLd}.
In analogy to \cref{prop:PalaisTW} we have the following \namecref{prop:PSCTWDiscrete}.

\begin{Proposition}[PSC for TW in discrete theory]\label{prop:PSCTWDiscrete}
Let $c \in \R \setminus \{0\}$, $b>0$, $X=([0,b/c]/\sim) \times ([0,b]/\sim)$ be a torus, $W=C^1(X,\R^d)$ continuously differentiable functions and $\Sigma = \{u \in W \,|\, u(t+s,x+cs)=u(t,x)\}$ be travelling waves with wave speed $c$.
Let $S \colon W \to \R$ be defined as in \eqref{eq:SLdCont}. Denote the restriction of $S$ to $\Sigma$ by $S_\Sigma \colon \Sigma \to \R$.
Then a travelling wave $u \in \Sigma$ is a stationary point of $S$, if and only if $u$ is a stationary point of $S_\Sigma$. 
\end{Proposition}

The proof is analogous to the proof of \cref{prop:PalaisTW} since the computation in \eqref{eq:ComputationPropPalaisTW} of \cref{sec:ProofPalaisTW} works for all autonomous Lagrangians. An analogous statement holds in the case of standing travelling waves $c=0$.

\begin{Corollary}[TW in discrete wave equation]
Travelling waves of the discrete wave equation (\cref{ex:DiscreteWave}) with periodic boundary conditions in space (period $b$) are governed by $S_\Sigma \colon \Sigma \to \R$ with 
\[
\begin{split}
S_\Sigma (f) &= \int_0^b \Bigg( \frac 12 \left( \frac {f(\xi + \Delta t) - f(\xi)}{\Delta t} \right)^2\\
&\phantom{ \frac bc \int_0^b}- \frac 12 \left( \frac {f(\xi + \Delta x) - f(\xi)}{\Delta x} \right)^2
- V(f(\xi)) \Bigg) \d \xi	
\end{split}
\]
using again the identification $\Sigma \cong C^1([0,b]/\sim,\R)$.
\end{Corollary}

\begin{Corollary}[TW in discrete Schrödinger]
Travelling waves of the discrete Schrödinger equation (\cref{ex:DiscreteSE}) are governed by $S_\Sigma \colon \Sigma \to \R$ with 
	\[
	S_\Sigma (f) = 
	\int_0^b L_d(f(\xi-\Delta t - \Delta x),f(\xi-\Delta t ),f(\xi),f(\xi-\Delta x)) \d \xi	
	\]
	and with $L_d$ of \eqref{eq:SELd}, using again the identification $\Sigma \cong C^1([0,b]/\sim,\R)$.
\end{Corollary}

The expressions $S_\Sigma (f)$ follows from a calculation analogous to \eqref{eq:ComputationPropPalaisTW}. In the case $c \not =0$ we have omitted the irrelevant constant factor $\frac bc$ when reporting $S_\Sigma$. 

\begin{figure}
	\includegraphics[width=0.4\linewidth]{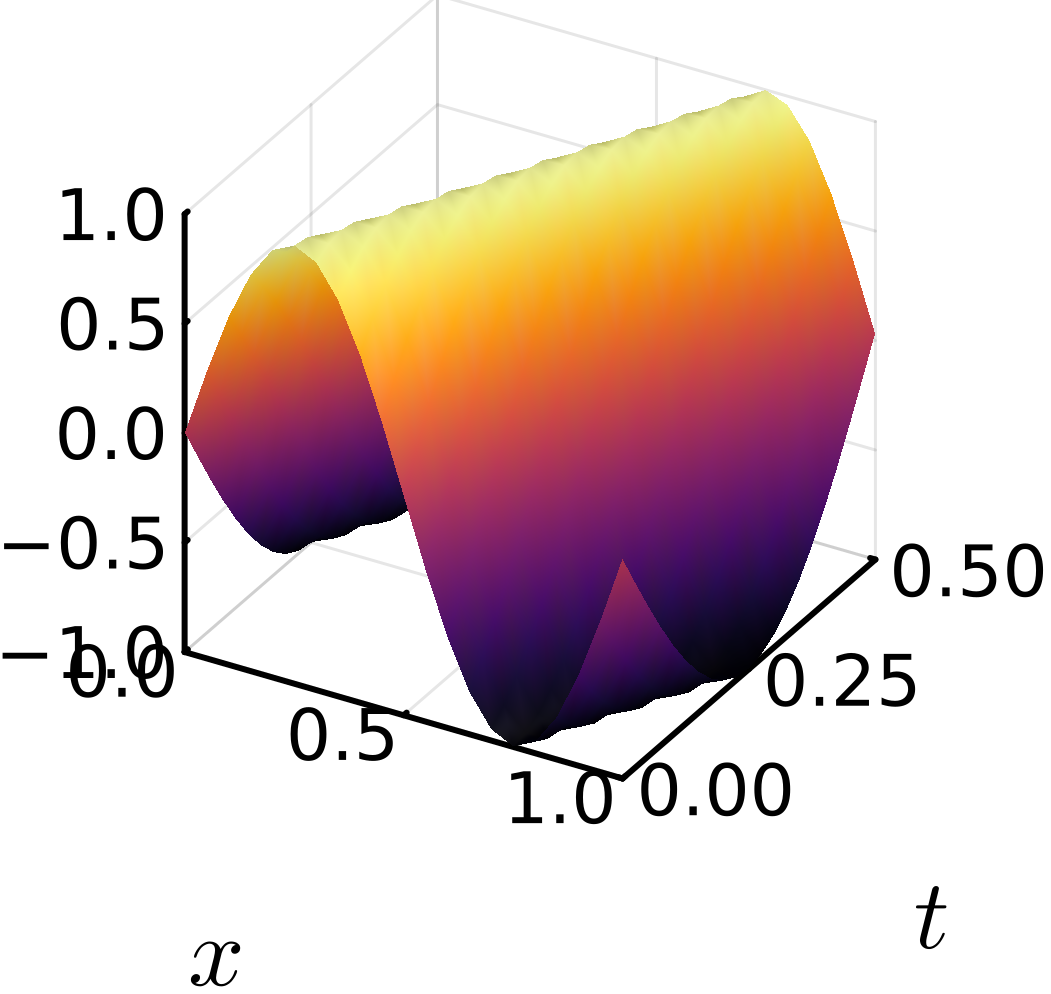}
	\includegraphics[width=0.4\linewidth]{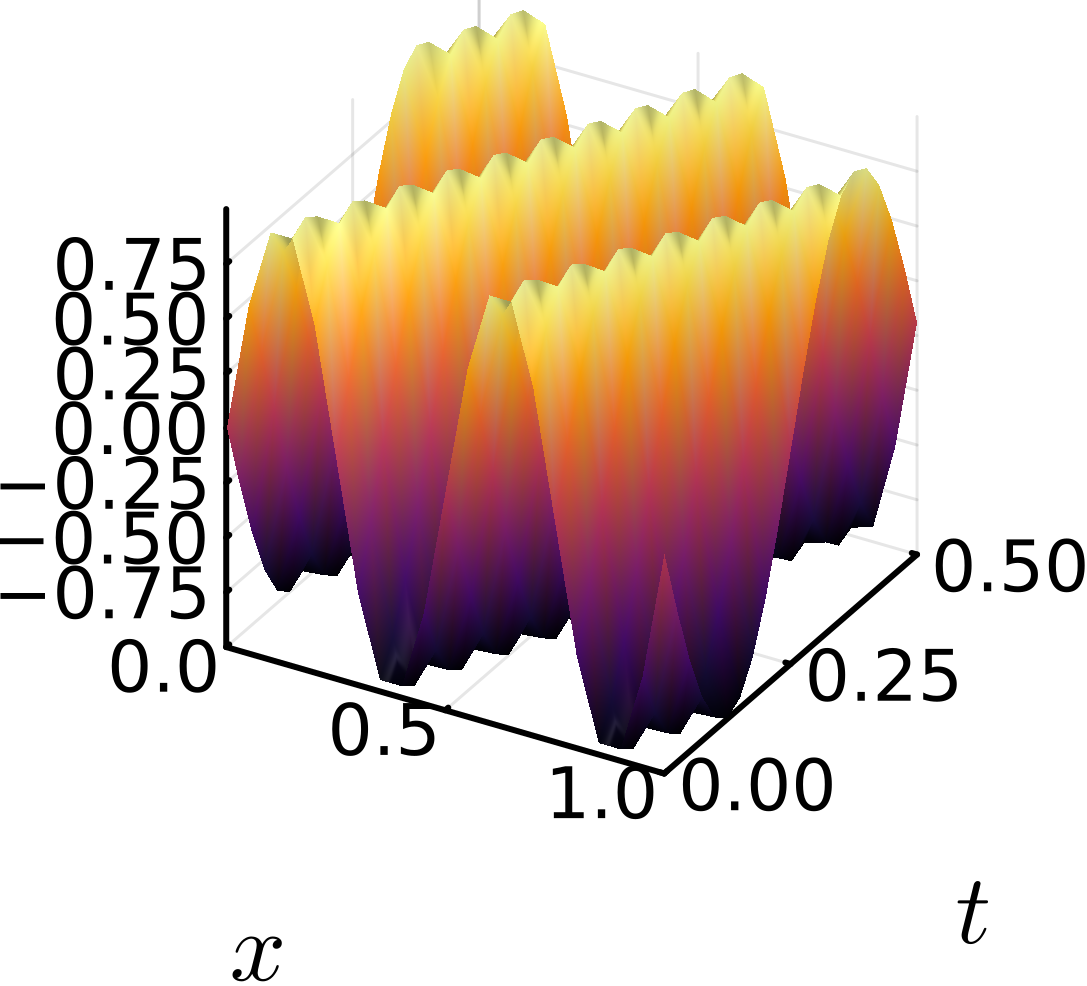}
	\caption{First two travelling waves ($m=1,2$ in \cref{rem:DiscreteWaveTW}) of the discrete wave equation with quadratic potential. The wave crest appears rough due to the mesh resolution. This is expected in discrete theories.}\label{fig:DiscreteWaveTW}
\end{figure} 
\begin{Example}[TW in discrete wave equation]\label{rem:DiscreteWaveTW}
Consider $V(u)=\frac 12 u^2$ in the discrete wave equation (\cref{ex:DiscreteWave}) with periodic boundary conditions in space (period $b$). A Fourier ansatz for $f$ reveals (away from some resonant or degenerate cases) that $u(t,x)  = f(\xi)= \alpha_1 \sin(\kappa \xi ) + \alpha_2 \cos(\kappa \xi)$, $\xi = x-ct$, $\kappa = 2 \pi m / b$, ($m \in \Z \setminus \{0\}$, $\alpha_1,\alpha_2 \in \R$) is a solution with wave speed $c$ a real solution of
\[
\cos( \kappa c \Delta t) = 1 - \frac{\Delta t^2}{2} + \frac{\Delta t^2}{\Delta x^2}(\cos(\kappa \Delta x)-1).
\]
See \cref{fig:DiscreteWaveTW} for plots of the first two travelling waves ($m=1,2$).
\end{Example}

\begin{figure}
\includegraphics[width=0.4\linewidth]{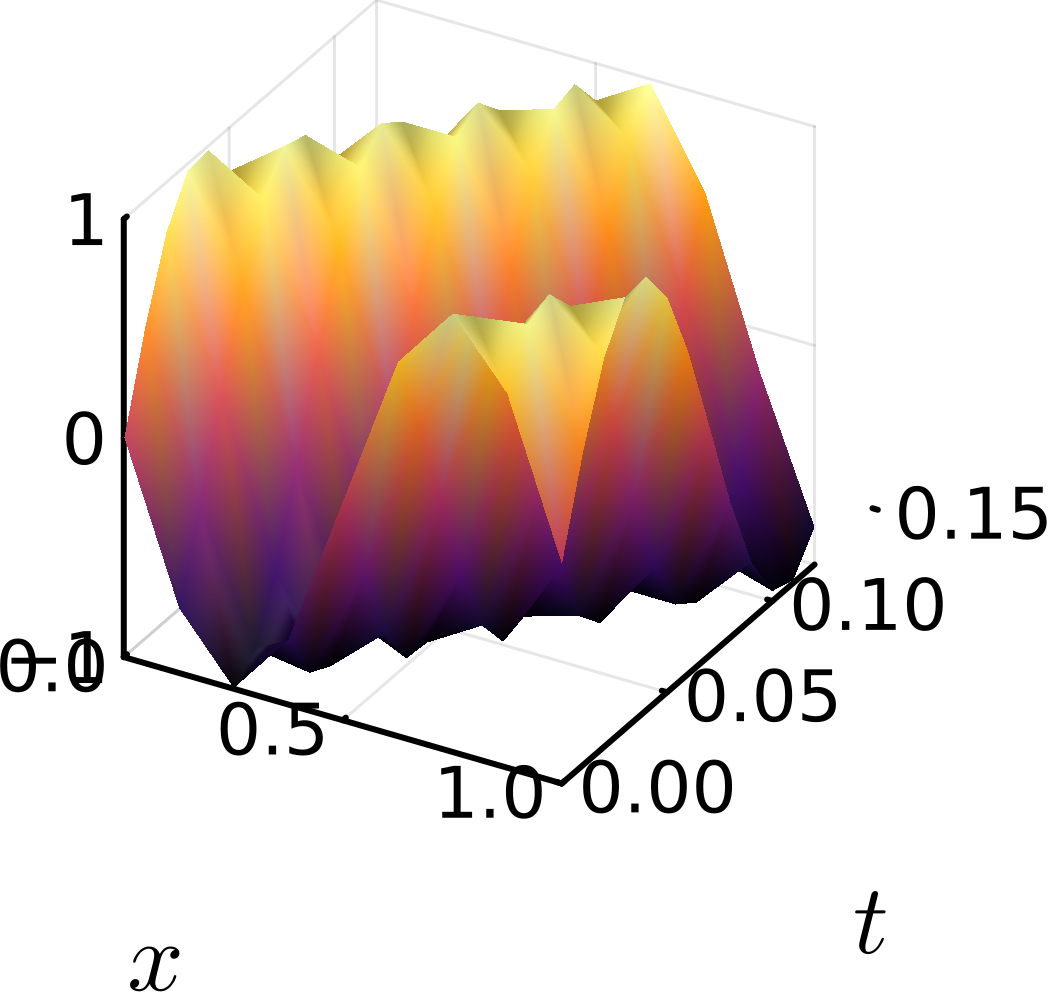}	
\includegraphics[width=0.4\linewidth]{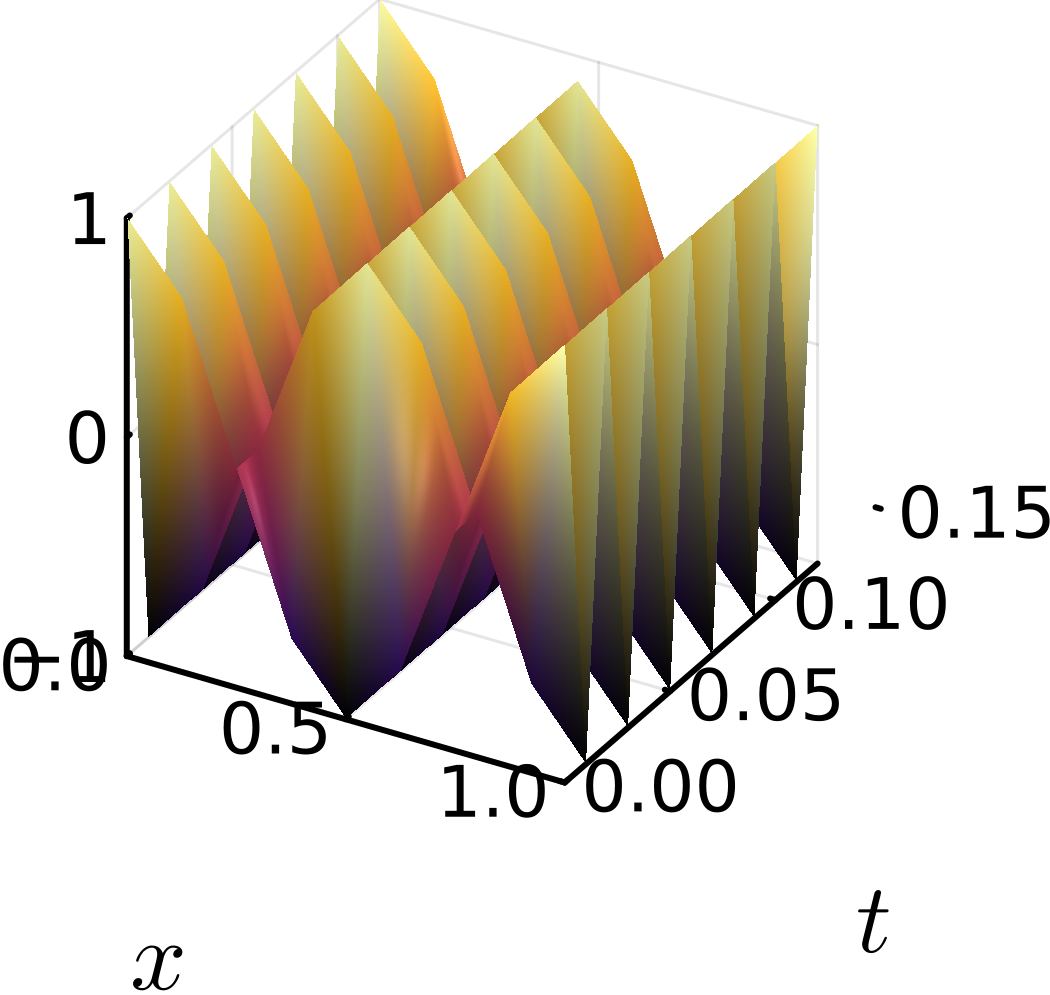}
\caption{Left: Travelling waves in \cref{rem:DiscreteWaveTWSE} of the discrete Schrödinger equation ($m=1$). The wave crest appears rough due to the mesh resolution. Right: Spurious travelling wave.}\label{fig:DiscreteWaveTWSE}
\end{figure} 
\begin{Example}[TW in discrete Schrödinger]\label{rem:DiscreteWaveTWSE}
Consider a linear potential $V(u)=\beta u$ in the discrete Schrödinger equation (\cref{ex:DiscreteSE}) with periodic boundary conditions in space (period $b$). A Fourier ansatz for $f$ reveals (away from some resonant or degenerate cases) that $u(t,x)  = f(\xi)= \alpha e^{\mathrm i \kappa \xi}$, $\xi = x-ct$, $\kappa = 2 \pi m / b$, ($m \in \Z\setminus\{0\}$, $\alpha \in \R$) is a solution with wave speed
\[
c = \frac{2}{\kappa \Delta t} \left( 
\arctan\left(
\frac{2}{\hbar} \frac{\Delta t}{\Delta x^2}
\tan^2(\frac 12 \kappa \Delta x)
+ \frac \beta {2 \hbar } \Delta t
\right) + s \pi
\right),
\]
where $s \in \Z$. Asymptotic expansion in $\Delta t, \Delta x$ recovers the expression \eqref{eq:WaveSpeedSE} up to higher order terms. Additional travelling wave solutions of the discrete theory are obtained with wave speed $c = \frac{b (2\tilde{m}+1) }{2 m \Delta t}$, $\tilde{m} \in \Z$. These are regarded as spurious since they do not converge to solutions of the continuous theory as $\Delta t \to 0$.
(See \cref{fig:DiscreteWaveTWSE}.)
\end{Example}

\begin{Remark}[Blended backward error analysis]
	Using the technique of blended backward error analysis\cite{BEAMulti,PDEBEA} we show that the discrete action $S_\Sigma$ can formally be written as a continuous first order theory $S_\Sigma(f)=\int \tilde L(f,f_\xi) \d \xi$, where $\tilde L$ is a formal power series in $\Delta t$, $\Delta x$.
\end{Remark}

%
%
%

\section{Numerical Experiments}\label{sec:Experiments}

We report numerical experiments based on learning discrete field theories for the discrete wave equation (\cref{ex:DiscreteWave}) and the discrete Schrödinger equation (\cref{ex:DiscreteSE}) from data. The discrete field theories with the discrete Lagrangian densities \eqref{eq:LdWave2d}, \eqref{eq:SELd} are regarded as the true theories rather than as approximations of continuous theories in this context.

The use-case of our machine learning architecture is to learn effective models of variational dynamical systems whose underlying dynamics is an unknown variational partial differential equation. The observational data is discrete and lies on a spacetime lattice. We, therefore, restrict ourselves to explore models which capture solutions whose discrete Fourier modes exhibit an exponential decay.
This disregards solutions that do not have a limit as lattice parameters are made asymptotically small. These are sometimes referred to as spurious solutions in numerical integration theory\cite{GeomIntegration} as they do not correspond to solutions of an underlying continuous theory.

We evaluate the performance of the data-driven models by computing or extending solutions of the trained models to initial data and checking consistency against the true models, compare performance to alternative reduced order modelling approaches (ROMs), and analyse how well travelling waves are captured. In the experiments, travelling waves will not be part of the training data but will be discovered using the trained model.

\subsection{Wave equation}\label{sec:WaveExperiment}

Consider the space-time domain $X=[0,T]  \times [0,b]/\sim$ (periodic boundary conditions in space), $T=0.5$, $b=1$, discretisation parameters $\Delta t = 0.025$, $\Delta x = 0.05$, and quadratic potential $V(u)=\frac 12 u^2$. The discrete lattice is denoted by $X_d$ and $M= b/\Delta x = 20$ is the number of spatial interior grid points and $N = T/\Delta t=20$ the number of time steps.

\subsubsection{Training data}\label{sec:TrainingDataWave}

\begin{figure}
	\includegraphics[width=0.4\linewidth]{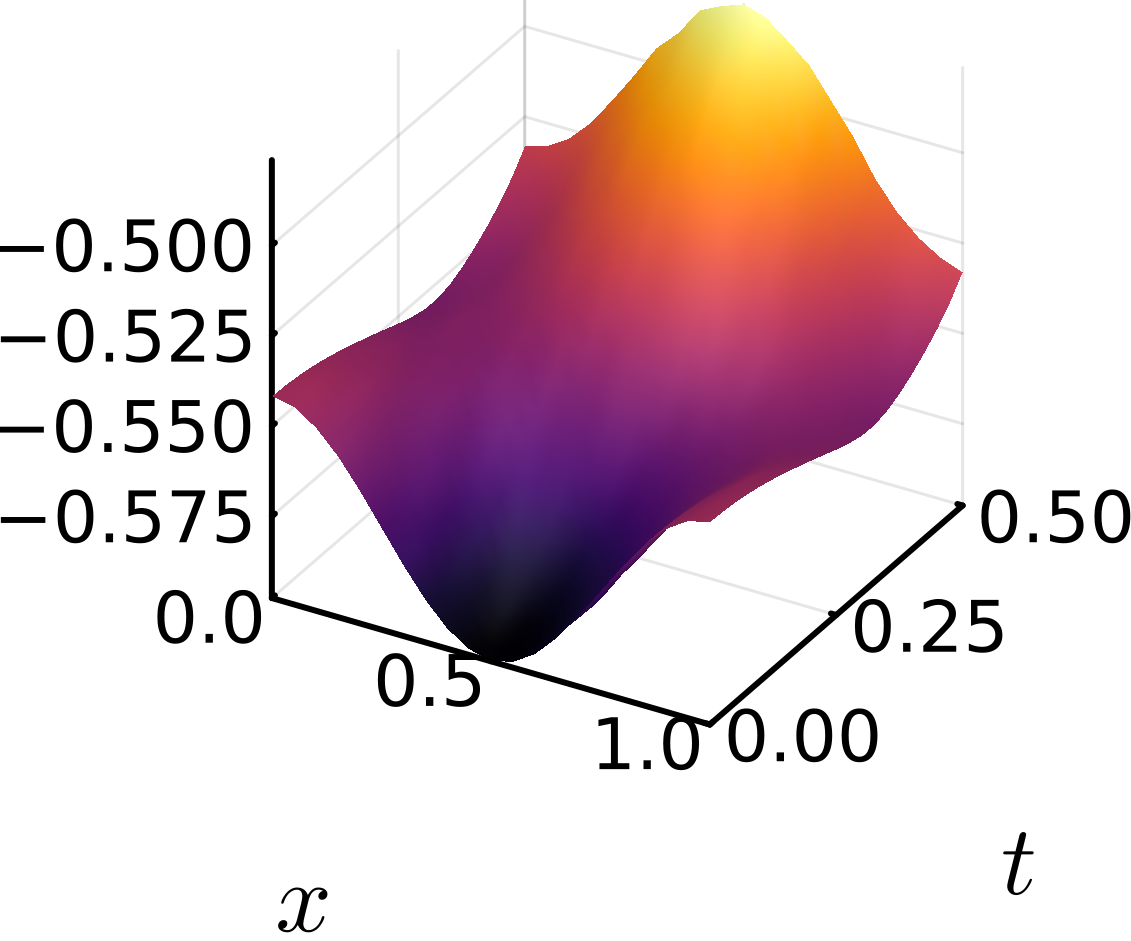}
	\includegraphics[width=0.45\linewidth]{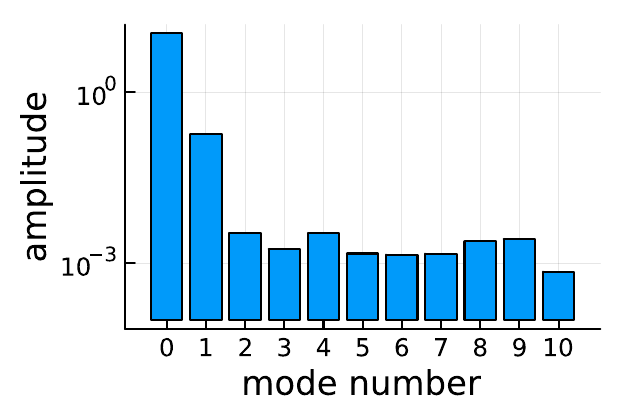}
		\includegraphics[width=0.4\linewidth]{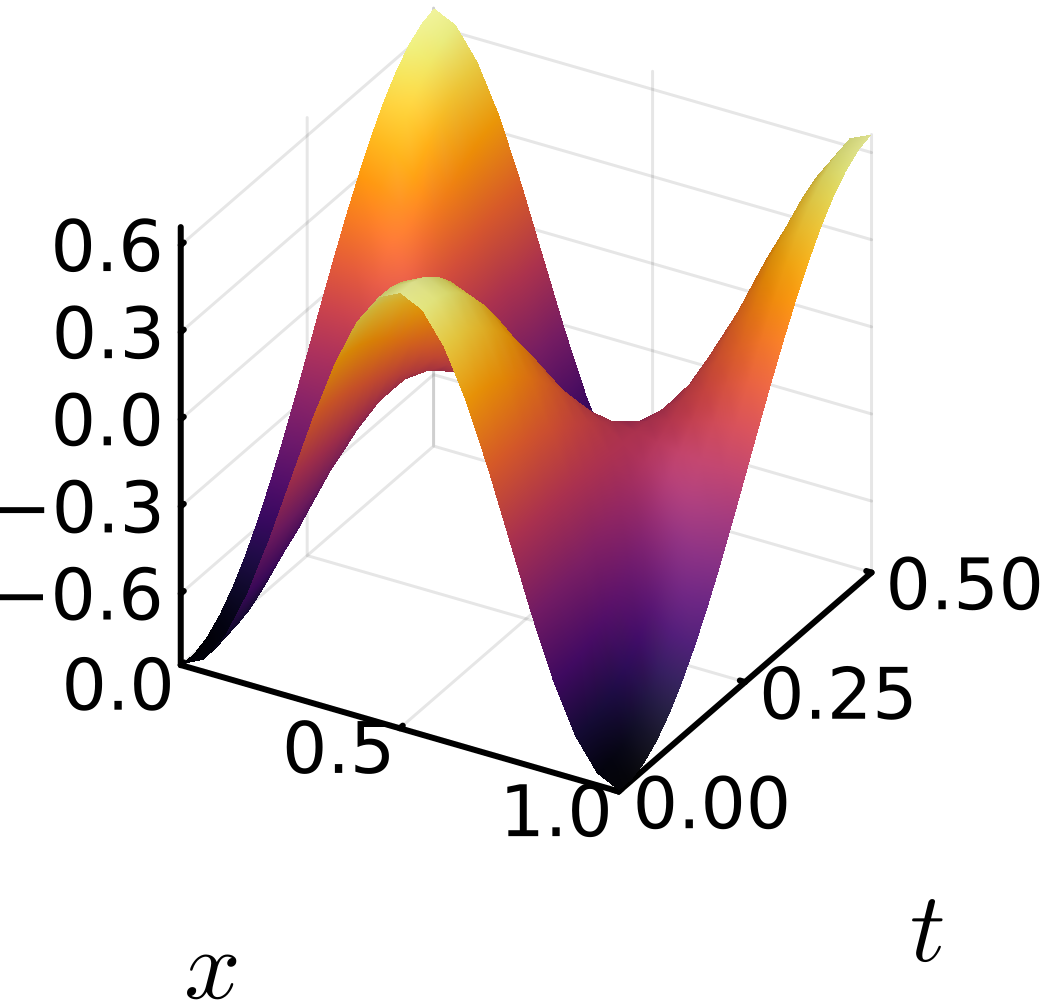}
	\includegraphics[width=0.45\linewidth]{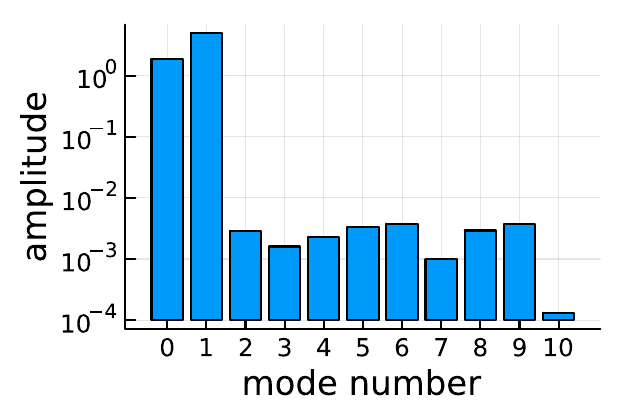}
	\includegraphics[width=0.45\linewidth]{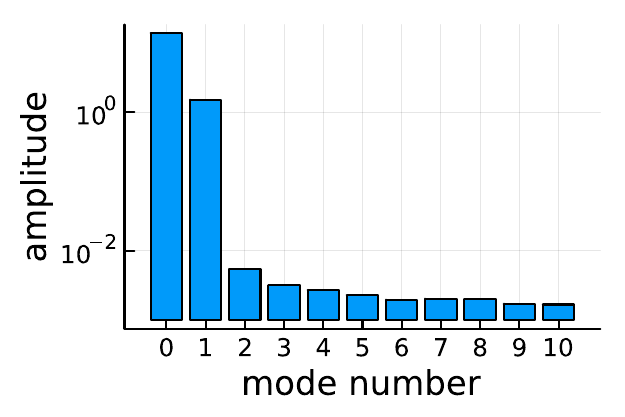}
	\caption{Typical elements in the training data set (left) and their absolute value of amplitudes of Fourier modes averaged over time (right). Bottom: Absolute values of amplitudes of Fourier modes averaged over time and all training data.}\label{fig:TrainingDataWave}
\end{figure}
To generate initial data, $M^{\mathrm{rfft}}=11$ Fourier coefficients $\gamma_j = M e^{-2 j^4} \eta_j$ ($j=0,\ldots,M^{\mathrm{rfft}}$) are computed where $\eta_j$ are samples from a standard normal distribution. An application of the inverse of the discrete Fourier transform for real input ($\mathtt{irfft}$ of Julia package FFTW.jl\cite{FFTWjl2005}) yields initial data $(u^0_j)_{j=0}^{M-1}$. Moreover, initial velocities $(\dot u^0_j)_{j=0}^{M-1}$ are drawn from a standard normal distribution. Then $K =80$ solutions of the field theory are computed (\cref{rem:comp34Ld}).
\Cref{fig:TrainingDataWave} displays two typical solutions and the absolute values of the amplitudes of their discrete Fourier modes averaged over all time steps. A corresponding average over all times and all solutions is displayed at the bottom of \cref{fig:TrainingDataWave}.
Each solution $U^{(k)}=(u^{i,(k)}_j)_{0\le j < M}^{0\le i \le N}$ yields $(N-1)M=380$ tuples $(u^{i,(k)}_{j},u^{i+1,(k)}_{j},u^{i,(k)}_{j+1},u^{i-1,(k)}_{j},u^{i-1,(k)}_{j+1},u^{i,(k)}_{j-1},u^{i+1,(k)}_{j-1})_{0 \le j < M}^{0 < i < N}$ of data corresponding to the $7$-point stencil. The 7-point tuples constitute the training data set $\mathcal{D}$.
$\mathcal{D}$ contains 30400 tuples of stencil data.
For batch training\cite{MarslandMachineLearning}, $\mathcal{D}$ is partitioned into batches of size 10.

\subsubsection{Training}

The 3 point discrete Lagrangian of \cref{ex:3ptLd} is modelled by a neural networks (3 layers, interior layer with 10 nodes, 160 parameters in total, activation function $\mathtt{tanh}$).
Consider $\ell_{\mathrm{data}}$ of \eqref{eq:elldata}
\begin{equation}\label{eq:elldataWave}
\ell_{\mathrm{data}} = \sum_{\mathrm{stencil\, data}\in \mathrm{batch\, of\,} \mathcal{D}} |\mathrm{DEL}(L_d)(\mathrm{stencil\, data})|^2
\end{equation}
with $\mathrm{DEL}(L_d)$ of \eqref{eq:DEL3pt}, 
 and $\ell_{\mathrm{reg}}$ from \eqref{eq:regLd3pt}
\begin{equation}\label{eq:regLd3ptWave}
\ell_{\mathrm{reg}}
=\alpha \sum \left\|\left(\frac{\p^2 L_d}{\p u^{i}_{j}\p u^{i+1}_j}({u^i_j}^{(k)},{u^{i+1}_j}^{(k)},{u^i_{j+1}}^{(k)})\right)^{-1}\right\|^2,
\end{equation}
with normalisation $\alpha =  \frac{1}{\mathrm{batch\, size}} = \frac 1 {10}$. The sum is taken over all stencil data $(u^{i,(k)}_{j},u^{i+1,(k)}_{j},u^{i,(k)}_{j+1},u^{i-1,(k)}_{j},u^{i-1,(k)}_{j+1},u^{i,(k)}_{j-1},u^{i+1,(k)}_{j-1})$ of a batch.

For each batch of $\mathcal{D}$ one optimisation step with the optimisation method $\mathtt{adam}$\cite{Adam} with the default learning parameters of the Julia/Flux.jl\cite{Fluxjl2018,innes2018} implementation is applied to the parameters of the neural network using the loss function $\ell = \ell_{\mathrm{data}}+\ell_{\mathrm{reg}}$. This is repeated 1320 times (number of epochs).
For the trained model we have $\ell_{\mathrm{data}} \approx 8.6 \cdot 10^{-8}$ and $\ell_{\mathrm{reg}} \approx 1.4 \cdot 10^{-7}$. (Here $\ell_{\mathrm{data}}$, $\ell_{\mathrm{reg}}$ are evaluated on the full data set $\mathcal{D}$.)

\subsubsection{Evaluation}

\paragraph{Extrapolation}

\begin{figure}
	\includegraphics[width=0.4\linewidth]{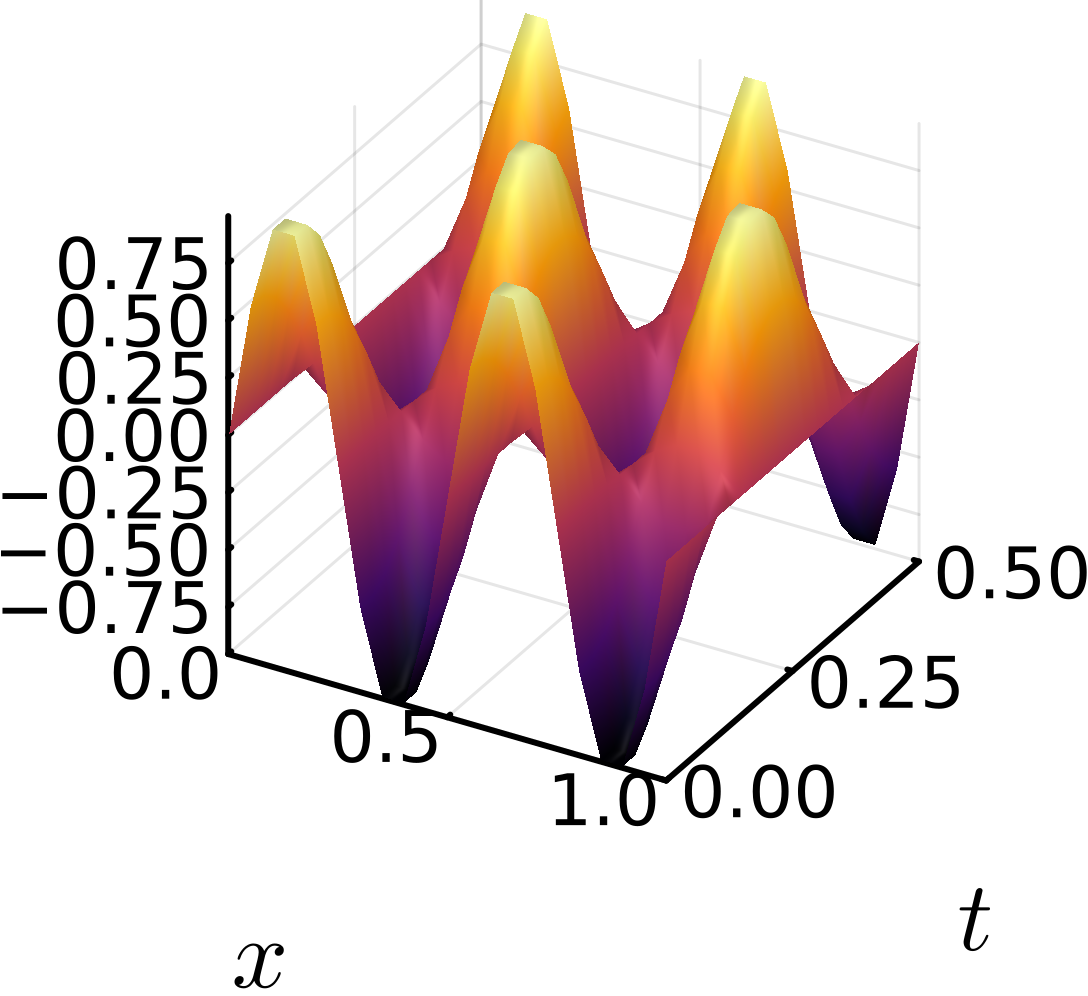}
	\includegraphics[width=0.4\linewidth]{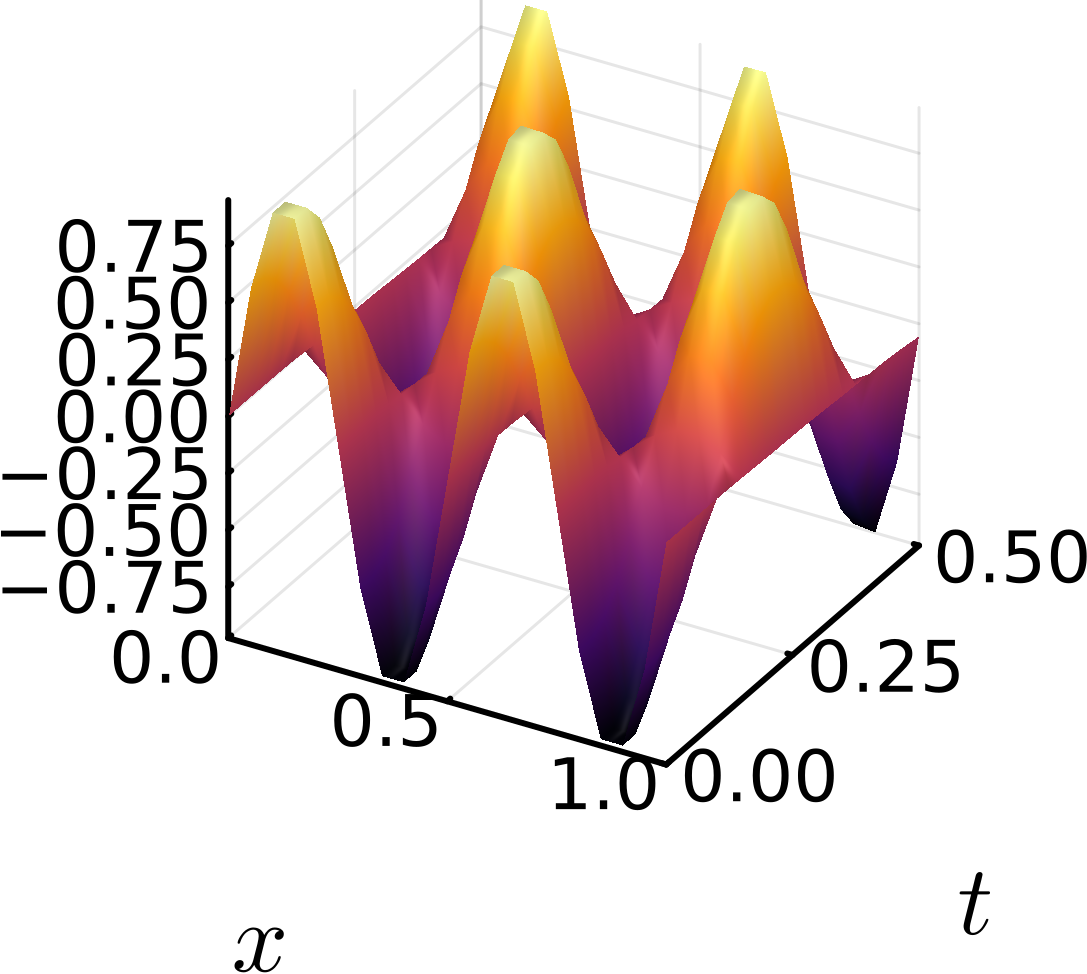}
	\includegraphics[width=0.4\linewidth]{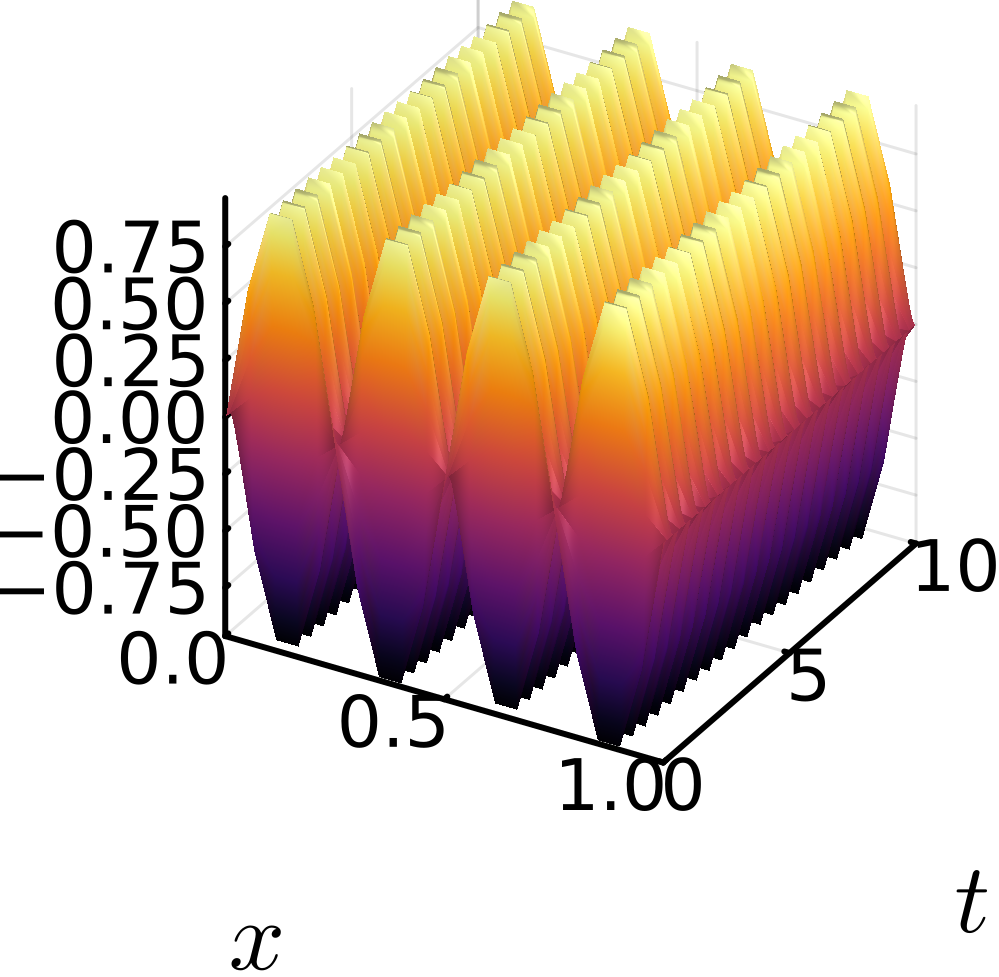}
	\includegraphics[width=0.4\linewidth]{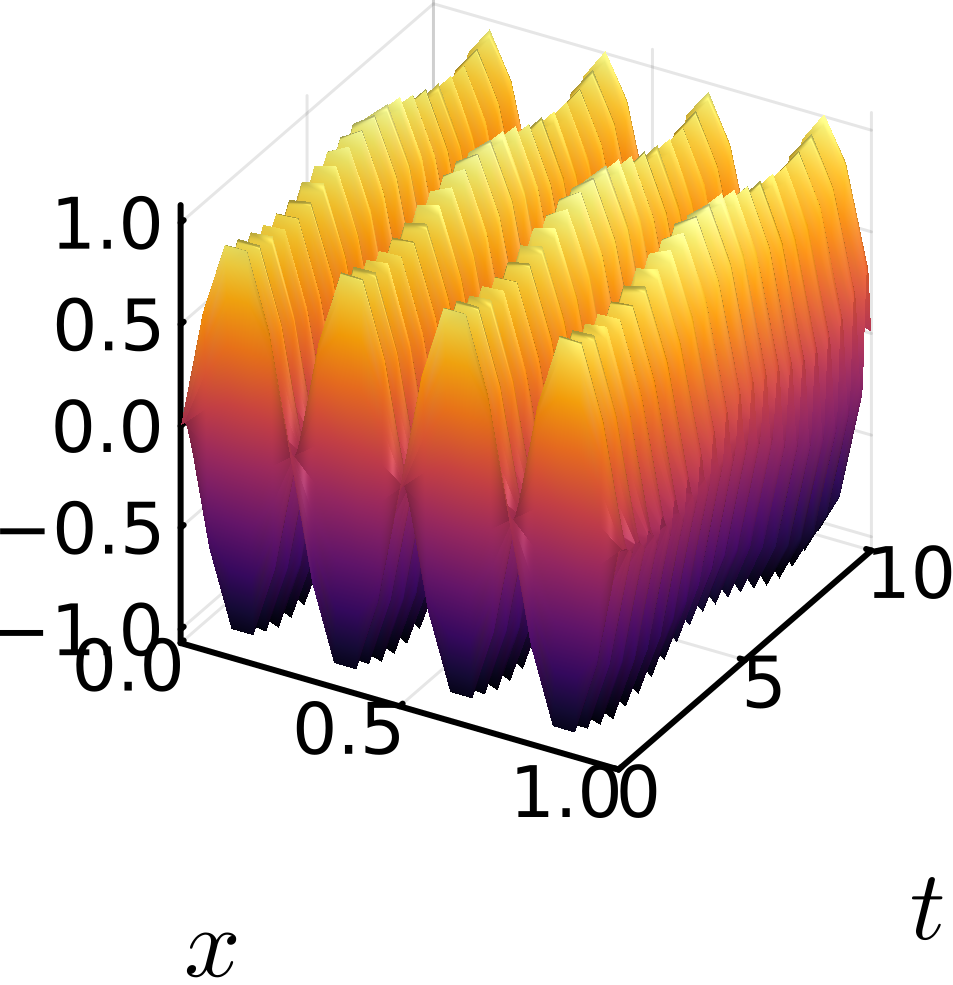}
	\caption{True solution of the model (left) versus computed solution from the learned model (right). The initial data $u^0_j = u^1_j=\sin(4\pi/T j \Delta x)$ was not seen during training. Top line shows predictions until $T=0.5$ (as in training), while the bottom line shows extrapolation until $20T=10$.}\label{fig:EvaluationWave}
\end{figure}
The trained neural network model of a discrete Lagrangian density is used to compute solutions of the learned field theory initialised from the unseen initial value $u^0_j = u^1_j = \sin(4\pi j \Delta x)$, $j=0,\ldots,M-1$ (\cref{fig:EvaluationWave}).
Computations are performed as explained in \cref{rem:comp34Ld}.
On the training domain the absolute difference between a true solution is $\| U_\mathrm{true} - U_\mathrm{predicted}\|_\infty \approx 0.057$. On the extrapolated domain with final time $20T$ the error rises to $\| U_\mathrm{true} - U_\mathrm{predicted}\|_\infty \approx 0.343$ and to $\approx 1.41$ for $100T$.

\paragraph{Travelling waves}\label{sec:TWExperimentWave}
The travelling waves of \cref{fig:DiscreteWaveTW} fulfil the discrete Euler--Lagrange equations \eqref{eq:DEL3pt} for the trained discrete Lagrangian density $L_d$ with a maximal error of $\max_{i,j} |\mathrm{DEL}(L_d)^i_j|< 0.004$ ($m=1$, left plot of \cref{fig:DiscreteWaveTW}) or $\max_{i,j} |\mathrm{DEL}(L_d)^i_j|< 0.033$ ($m=2$, right plot of \cref{fig:DiscreteWaveTW}), respectively.
Therefore, these true travelling waves constitute solutions of the learned discrete field theory (up to small error).
Travelling waves can be found numerically in the trained model as well: for a guess of a wave speed $c$ and a periodic function $f \colon [0,b]/\sim \to \R$
\begin{equation}\label{eq:fTWFinder}
f(\xi) = \sum_{m=- (1+\lfloor \frac M2 \rfloor)}^{1+\lfloor \frac M2 \rfloor} \hat{f}_{|m|} e^{\frac{2\pi \mathrm{i}m}{b} \xi}
\end{equation}
represented by $1+\lfloor \frac M2 \rfloor$ Fourier coefficients $\hat{f}_{m} \in \mathbb{C}$, the objective function
\[
\ell_{\mathrm{wave}}(c,(\hat{f})_{m=0}^{1+\lfloor \frac M2 \rfloor})
= \sum_{i=1}^{N-1}\sum_{j=0}^{M-1} \| \mathrm{DEL}(L_d)^i_j\|^2
\]
is minimised with $\mathrm{DEL}(L_d)^i_j$ from \eqref{eq:DEL3pt} with $u^i_j = f(i \Delta t - c j\Delta x)$ for all $i,j$.
To avoid finding the trivial solution, the normalisation condition $\ell_{\mathrm{wave}}^\mathrm{unit} = |\| f\|^2_{L^2}-1|$ is added to the objective. Here $\| f \|^2_{L^2} = \Delta x \sum |\hat{f}_{|m|}|^2$ is the discrete $L_2$ norm of $f$.

For the experiment, we initialise the optimisation procedure from the true Fourier coefficients $(\hat{f}_m)_m$ and wave speed $c$ corresponding to the true travelling waves of \cref{rem:DiscreteWaveTW} (wave number $m=1,2$) perturbed by random noise (identically independently normally distributed, standard deviation $\sigma = 0.5$). The located travelling wave $U^{\mathrm{TW}}$ has maximal norm error $\|U^{\mathrm{TW}} - U^{\mathrm{TW}}_{\mathrm{ref}}\|_\infty <0.116$ (wave number $m=1$) or $\|U^{\mathrm{TW}} - U^{\mathrm{TW}}_{\mathrm{ref}}\|_\infty <0.31$ (wave number $m=2$), respectively. The wave speeds are found up to an error of $|c^{\mathrm{TW}}-c^{\mathrm{TW}}_{\mathrm{ref}}| <0.0009$ or $|c^{\mathrm{TW}}-c^{\mathrm{TW}}_{\mathrm{ref}}| <0.0082$, respectively. \Cref{fig:TWLocate} compares the contours of the located and the true discrete travelling wave for the case $m=2$.
\begin{figure}
	\includegraphics[width=0.45\linewidth]{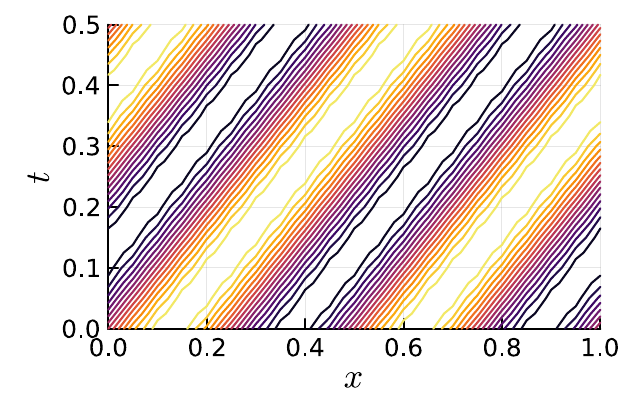}
	\includegraphics[width=0.45\linewidth]{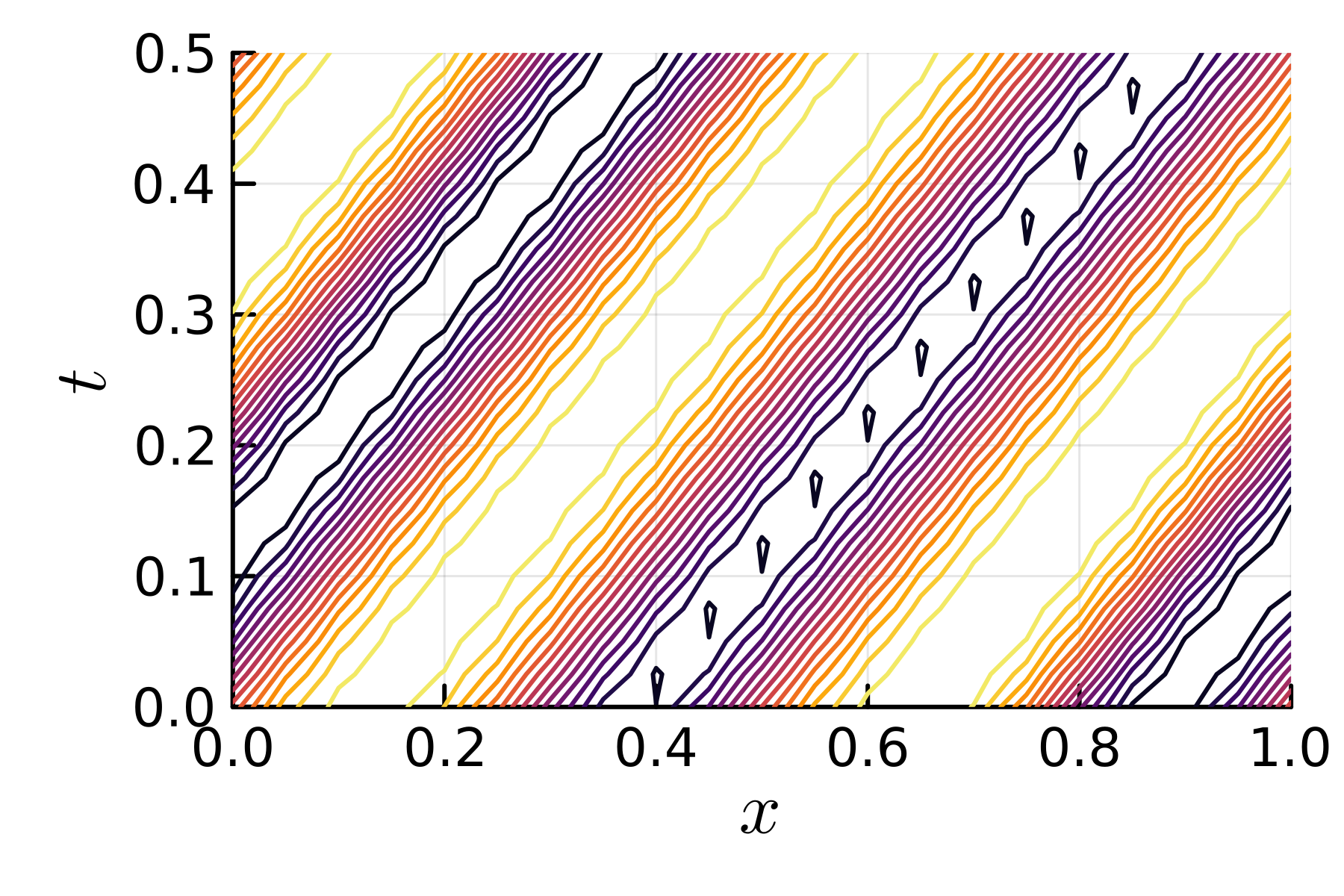}
	\caption{Contour plot of true travelling wave (left) with wave number 2 and the located travelling wave (right). }\label{fig:TWLocate}
\end{figure}

\subsection{Coarse Mesh}
Even though training data $(U^{(k)})_{k=1}^K$ (see \cref{sec:TrainingDataWave}) is available on the fine mesh $X_d$ with $M(N+1)$ mesh points, the discrete model $L_d$ can be trained on a coarser mesh. This can be useful to save computational costs when only low resolution predictions are required. Consider the mesh $X_d^\mathrm{coarse}$ for $X$ with $(\Delta t^{\mathrm{coarse}},\Delta x^{\mathrm{coarse}})=(2\Delta t, 2\Delta x)$. We find $(N-3)*M*K = 27200$ tuples $(u^{i,(k)}_{j},u^{i+2,(k)}_{j},u^{i,(k)}_{j+2},u^{i-2,(k)}_{j},u^{i-2,(k)}_{j+2},u^{i,(k)}_{j-2},u^{i+2,(k)}_{j-2})_{0 \le j < M}^{0 < i < N-1}$ 
with data over the coarse 7-point stencil in the observational data. Notice that this makes efficient use also of data over mesh points in $X_d$ that do not lie in $X_d^{\mathrm{coarse}}$. 
Let $\mathcal{D}^\mathrm{coarse}$ be a set of the 27200 training data tuples.
The discrete Lagrangian density $L_d^{\mathrm{coarse}}$ is parametrised with a neural network of the same architecture as before. For training, the network parameters are randomly initialised. The same expressions for the loss function \eqref{eq:elldataWave}, \eqref{eq:regLd3ptWave} and the same batching and optimisation methods are used. 
$L_d$ is trained for 63800 epochs with final losses $\ell_{\mathrm{data}} \approx 1.1 \cdot 10^{-5}$, $\ell_{\mathrm{reg}} \approx 2.5 \cdot 10^{-5}$.

\begin{figure}
	\includegraphics[width=0.4\linewidth]{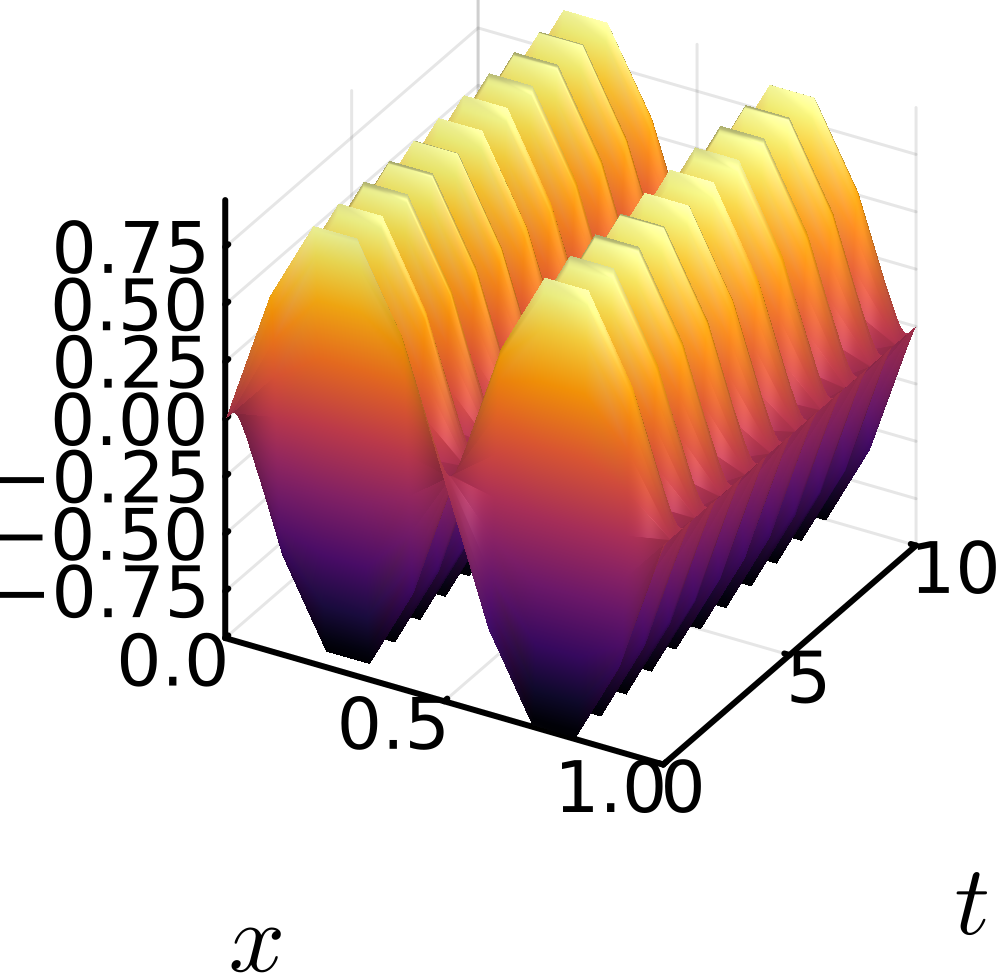}
	\includegraphics[width=0.4\linewidth]{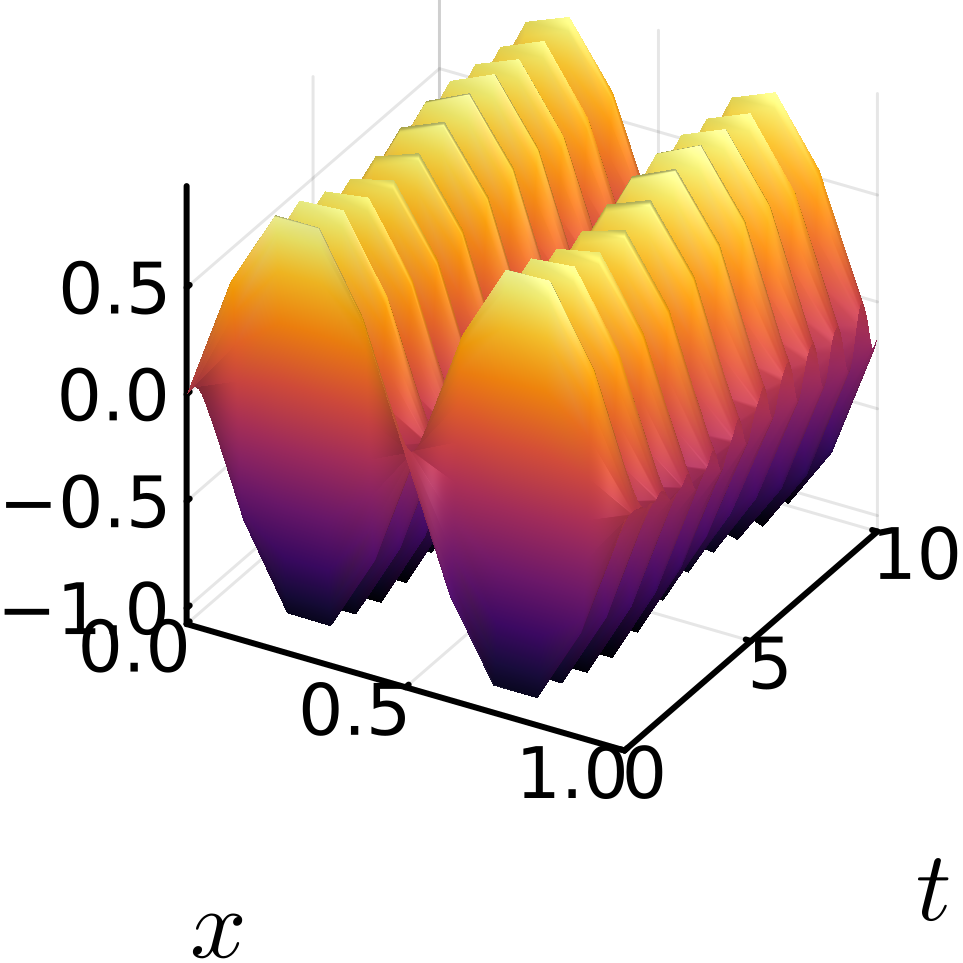}
	\includegraphics[width=0.4\linewidth]{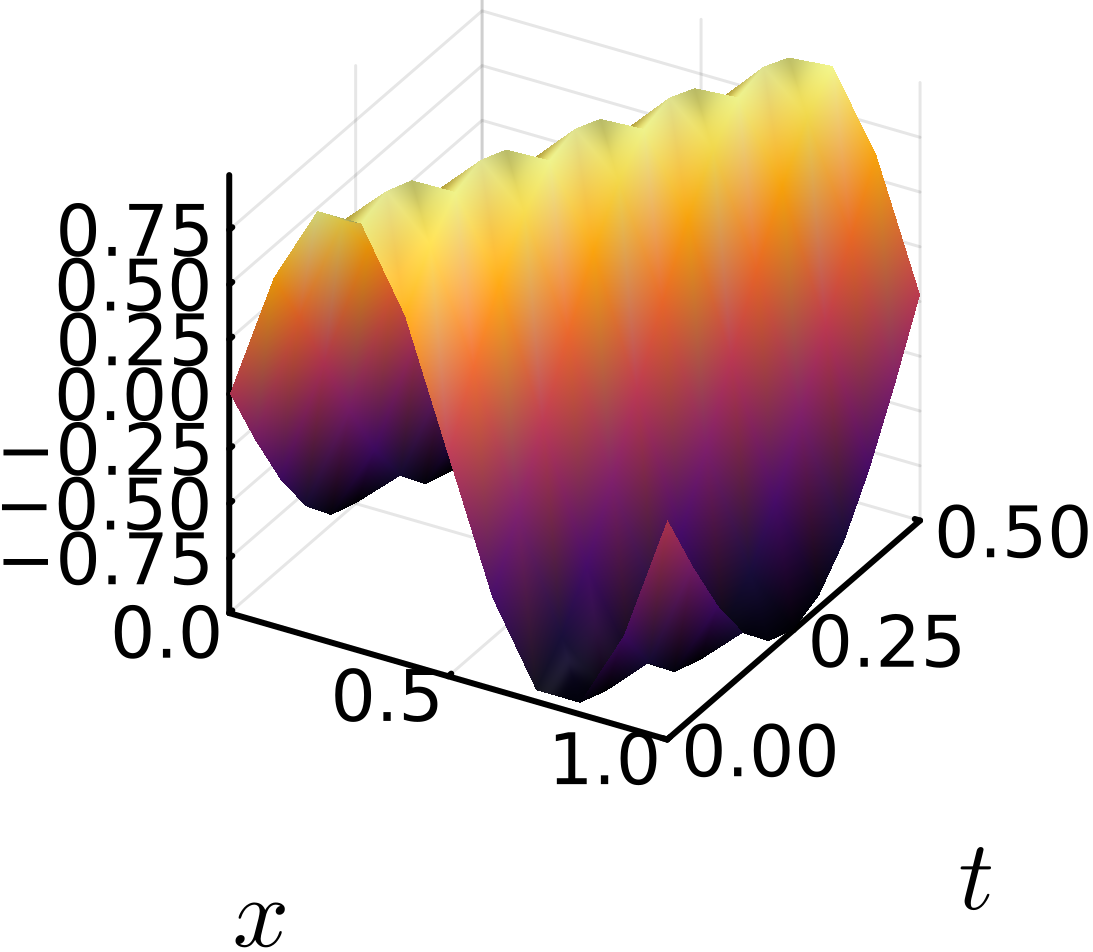}
	\includegraphics[width=0.4\linewidth]{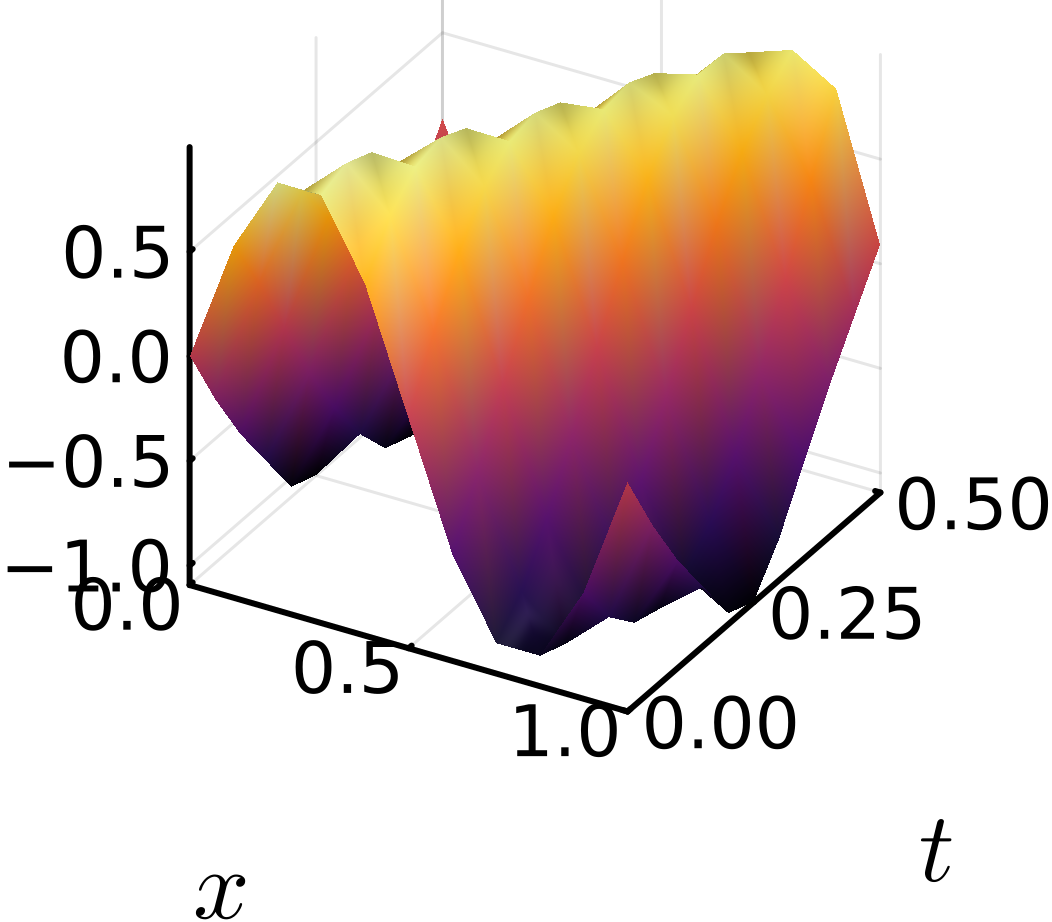}
	\caption{Model for discrete wave equation trained for a coarser grid. Top: Computation of solution to unseen initial data $u^0_j = u^1_j=\sin(2\pi/T j \Delta x^{\mathrm{coarse}})$ on extended domain $20T$ (left: reference, right: data-driven model).
	Bottom: Initialisation with exact data $u^0_j$, $u^1_j$ of a travelling wave. (left: reference, right: data-driven model)
	}\label{fig:CoarseEvaluation}
\end{figure}
\Cref{fig:CoarseEvaluation} shows numerical experiments with the trained coarse model.
In the top row of the figure, the (coarse) 7-point stencil is used to compute a solution $U_{\mathrm{predicted}}$ on $X_d^{\mathrm{coarse}}$ initialised from the unseen initial data $u^0_j = u^1_j=\sin(2\pi/T j \Delta x^{\mathrm{coarse}})$.
The prediction contains an extrapolation until $20T=10$.
A reference is computed with the 7 point stencil (reducing to a 5 point stencil) w.r.t.\ the original, fine mesh $X_d$ applied to the discrete Lagrangian density \eqref{eq:LdWave2d} of \cref{ex:DiscreteWave} (the true model). 
Subsampling over $X_d^\mathrm{coarse}$ yields the reference $U_{\mathrm{true}}$.
We obtain $\|U_{\mathrm{predicted}}-U_{\mathrm{true}}\|_\infty \approx 0.33$.
The bottom row of \cref{fig:CoarseEvaluation} shows the prediction $U_\mathrm{predict}^{\mathrm{TW}}$ of the model when initialised from data $u^0_j = u^1_j$ of a travelling wave $(m=1)$ of the true system.
We have $\|U_\mathrm{predict}^{\mathrm{TW}} - U_\mathrm{true}^{\mathrm{TW}}\|_\infty \approx 0.24$.

\subsection{Comparison to MOR based technique}

For comparison, we employ model order reduction in the spatial variable and then learn the temporal dynamics. This approach was explored in a recent articles in the continuous Lagrangian setting\cite{sharma2022Lagrangian} and in related settings and problems\cite{Blanchette2020,Blanchette2022,Glas2023,Carlsberg2015,sharma2023symplectic,Sharma2022,Tyranowski2022}. However, for consistency we will continue using discrete theory and continue using the training data of the discrete wave equation model (\cref{sec:TrainingDataWave}).

We use standard principle component analysis\cite{BruntonKutz2019,trefethen97} to identify a linear projection map $\mathrm{pr} \colon \R^{M} \to \R^{M_\mathrm{red}}=:Q$ and a recovery map $R \colon Q \to \R^M$ such that $\mathrm{pr} \circ R = \mathrm{id}_{Q}$ with $M_\mathrm{red} < M$. Then a neural network model of a discrete Lagrangian $L^Q_d \colon Q \times Q \to \R$ is trained to recover the temporal dynamics.
To identify $\mathrm{pr}$, let
\[
\mathcal{U}^\mathrm{space} 
= \begin{pmatrix}
U^{(1)}, & U^{(2)},& \ldots,& U^{(K)}
\end{pmatrix} \in \R^{M K(N+1)}
\]
denote a matrix arranging the matrices $U^{(k)} = (u^i_j)_{0 \le i \le N}^{0 \le j < M}$ of training data horizontally. Consider its singular value decomposition $\mathcal{U}^\mathrm{space} = A \Sigma B^\top$. Let $A' = \begin{pmatrix}
a_1,& a_2
\end{pmatrix} \in \R^{M \times M_\mathrm{red}}$
contain the first $M_\mathrm{red}$ columns of $A$ and set $\mathrm{pr}(u) = {A'}^\top u$ and $R(q) = A' q$. For $M_\mathrm{red}=2$ the average relative reconstruction error is
\[
\frac{1}{MK(N+1)} \sum_{k=1}^{MK(N+1)} \frac{\| \mathfrak{u}_k - R(\mathrm{pr}(\mathfrak{u}_k))\|_2}{\| \mathfrak{u}_k \|_2}
\approx 0.0026,
\]
where $\mathfrak{u}_k$ denote the columns of $\mathcal{U}^\mathrm{space}$.

To learn the temporal dynamics on the latent space, consider a neural network model of $L^Q_d \colon Q \times Q \to \R$ (3 layers, interior layer of size 10, 170 parameters, activation function $\mathtt{softplus}$).
Let $q^{i, (k)} \in Q$ denote the projected training data $q^{i}_{(k)}=\mathrm{pr}( ({u^i_j}^{(k)})_{0 \le j <M} )$ ($1 \le k \le K, 0\le i \le N$) and let 
\[
\begin{split}
\mathcal{D}^Q
=\big\{ (q^{i-1}_{(k)},q^{i}_{(k)},q^{i+1}_{(k)}) &\, | \, 1 \le k \le K, 0<i<N  \big\}
\end{split}
\] be the collection of tipples in the projected training data. Consider the loss function $\ell^Q = \ell_{\mathrm{data}}^Q + \ell_{\mathrm{reg}}^Q$ with 
\begin{equation}
\ell_{\mathrm{data}}^Q = \sum_{ q_{\mathrm{stencil}} \in \mathcal{D}^Q} \mathrm{DEL}(L_d^Q)(q_{\mathrm{stencil}})^2
\end{equation}
with 
\begin{equation}\label{eq:DELLd2pt}
\mathrm{DEL}(L_d^Q)(q_{\mathrm{stencil}})
= \frac{\p}{\p q^i} \left( L_d^Q(q^{i-1},q^i) + L_d^Q(q^i,q^{i+1}) \right)
\end{equation}
where $q_{\mathrm{stencil}}=(q^{i-1},q^i,q^{i+1})$ and regularisation
\[
\ell_{\mathrm{reg}}^Q = \sum_{ q_{\mathrm{stencil}} \in \mathcal{D}^Q} 
\left\|\left(\frac{\p^2 L_d^Q}{\p q^i \p q^{i+1}}L_d^Q(q^i,q^{i+1})\right)^{-1}\right\|^2.
\]
The regularisation is an analogy to \eqref{eq:regLd3pt} for discrete Lagrangians of two input arguments optimised for forward propagation by solving \eqref{eq:DELLd2pt} for $q^{i+1}$. The summands are approximated with 3 inverse vector iterations (see \cref{rem:ApproxRho}).
Again, batch training is performed with a batch size of 10. The regularisation term $\ell_{\mathrm{reg}}^Q$ was scaled by $10^{-8}$ such that both terms of the loss function are of the same order of magnitude at the beginning of training with randomly initialised parameters of the neural network. After 54750 epochs of optimisation with $\mathtt{adam}$, $\ell_{\mathrm{data}}^Q < 1.19 \cdot 10^{-5}$, $\ell_{\mathrm{reg}}^Q < 1.54 \cdot 10^{-5}$.

\begin{figure}
	\includegraphics[width=0.45\linewidth]{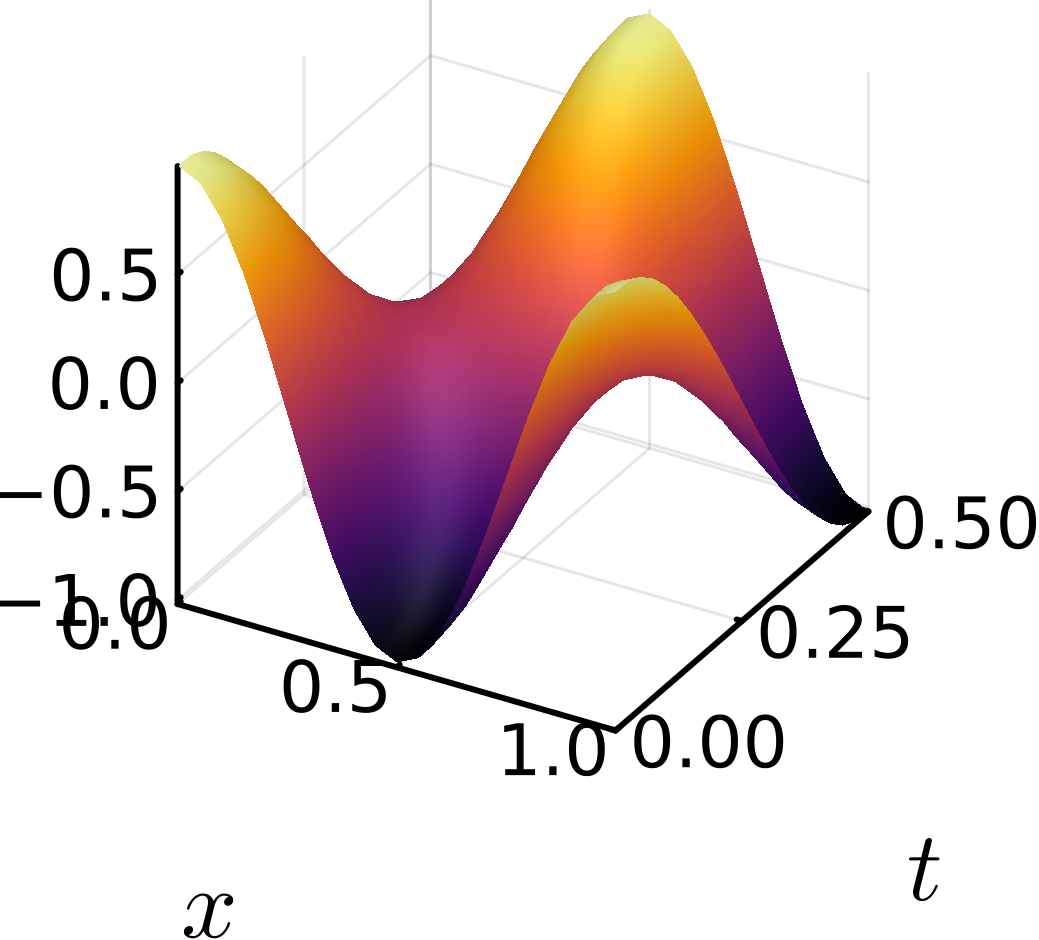}
	\includegraphics[width=0.45\linewidth]{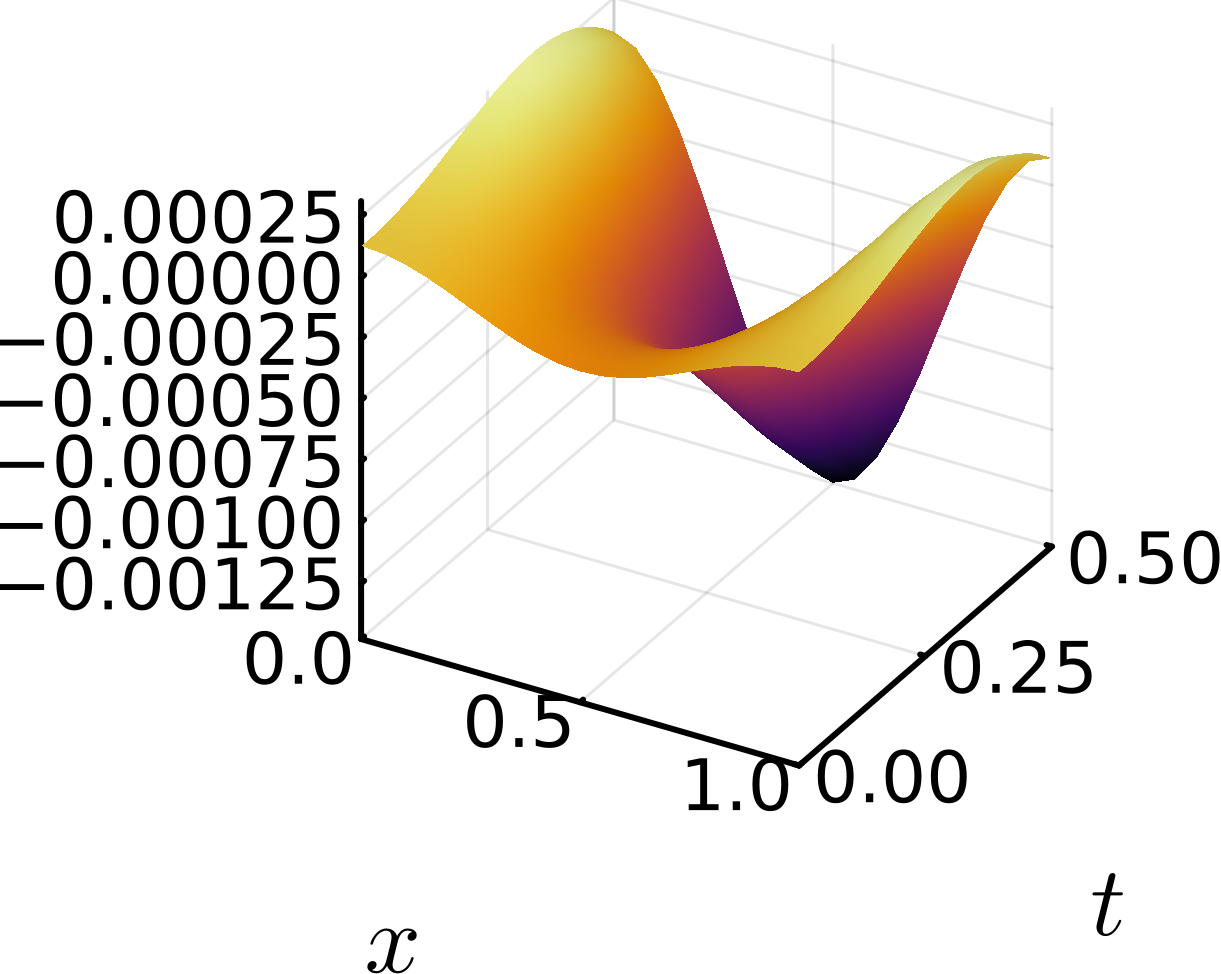}
	\includegraphics[width=0.45\linewidth]{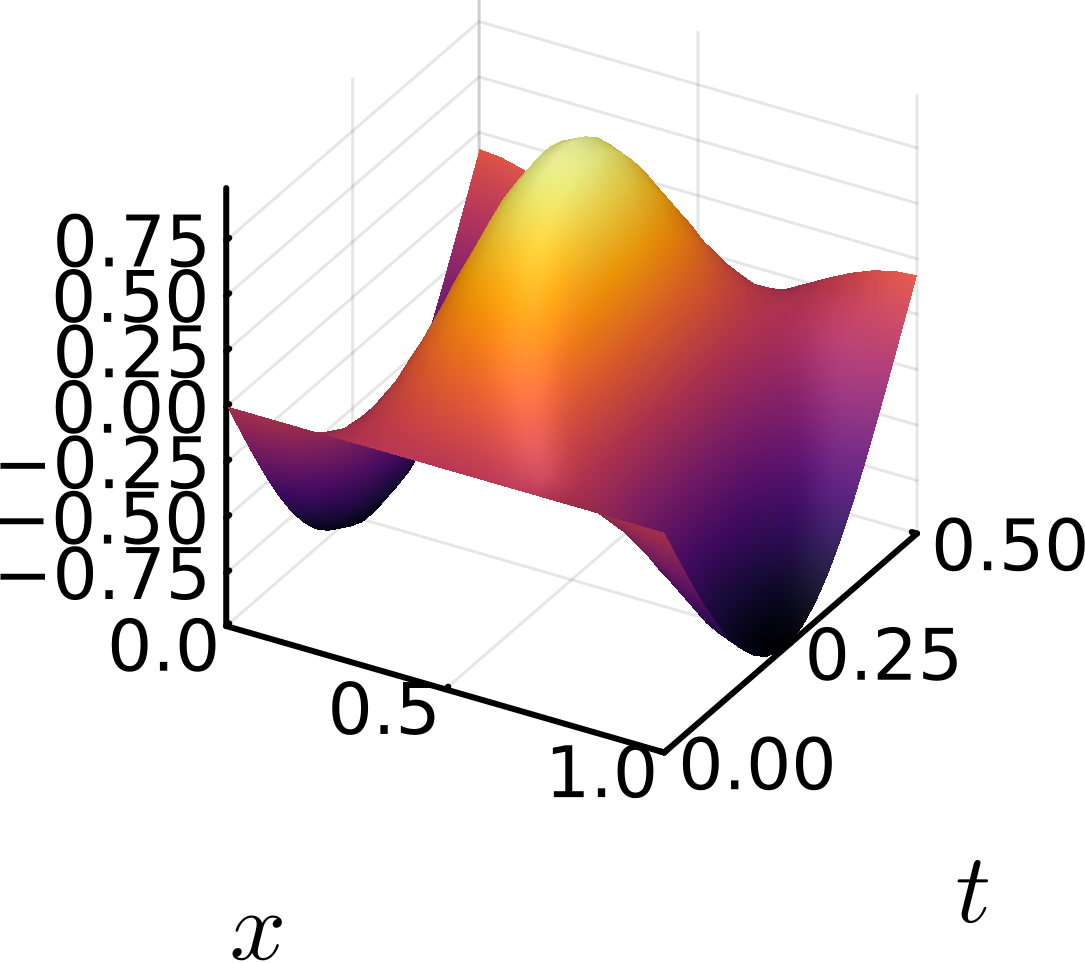}
	\includegraphics[width=0.45\linewidth]{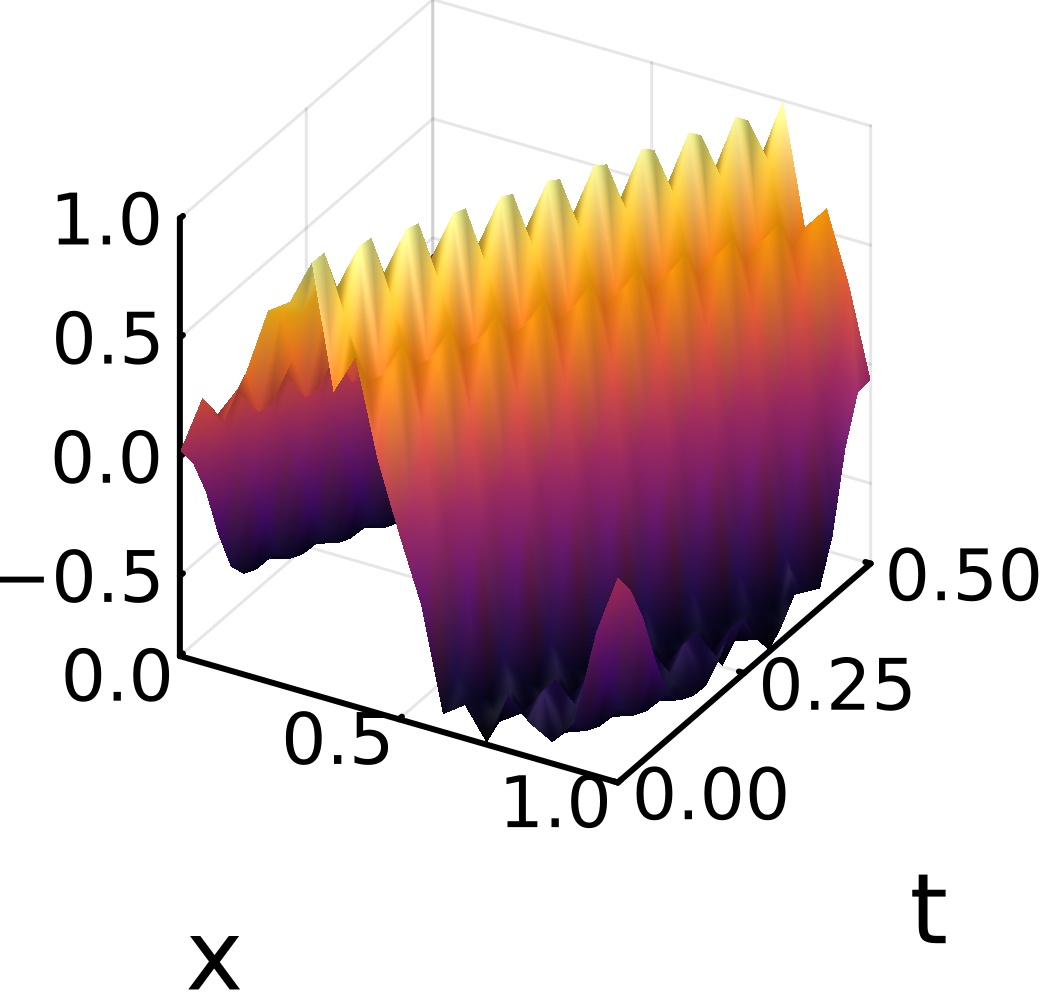}
	
	\caption{While the ROM model successfully predicts motions initialised close to training data (top left) it struggles to generalise (top right compare to top of \cref{fig:EvaluationWave}). Moreover, initialising with data of a true travelling wave does not produce a travelling wave (bottom left), while spurious travelling waves fulfil the discrete Euler--Lagrange equations when projected to the latent space (bottom right). }\label{fig:ROMExperiments}
\end{figure}
The trained model can recover the training data shown in \cref{fig:TrainingDataWave} in max-error norm $\| \cdot \|_\infty$ up to 0.003 or 0.006 when initialised with the exact data at time $t=0$, $t=\Delta t$. 
\Cref{fig:ROMExperiments} shows that while the ROM model successfully predicts dynamics with initial data close to training data (top left), it fails on initial data such as $u^0_j = u^1_j = \sin(4\pi j \Delta x)$, $j=0,\ldots,M-1$ (top right). This can be compared to the proposed discrete field theory model, which shows good performance on this task (see \cref{fig:EvaluationWave}).

Unsurprisingly, as the training data does not contain any travelling waves, it fails to produce a good prediction when initialised with the exact initial data of a travelling wave. Indeed, the reconstructed solution has no similarity to a travelling waves (\cref{fig:ROMExperiments} bottom left). In contrast, the proposed stencil based scheme produces an output that is visually undistinguishable (not displayed) from the exact travelling wave (left of \cref{fig:DiscreteWaveTW}).

The travelling wave locator for discrete field theories described in \cref{sec:TWExperimentWave} can easily be adapted to the ROM setting: as a loss function we use
\[
\ell_{\mathrm{wave}}^Q(c,(\hat{f})_{m=0}^{1+\lfloor \frac M2 \rfloor})
=
\sum_{i=1}^{N-1} \left\|\frac{\p}{\p q^i} \left( L_d^Q(q^{i-1},q^i) + L_d^Q(q^i,q^{i+1}) \right) \right\|^2
\]
with $q^i = \mathrm{pr}( (u^i_j)_{0\le j < M} )$, with $u^i_j = f(i \Delta t - c j\Delta x)$ for all $i,j$, and $f$ as in \eqref{eq:fTWFinder}.
We obtain the result $U^{\mathrm{TW}}_Q$ to the bottom right of \cref{fig:ROMExperiments}: indeed, the projection $\mathrm{pr}(U^{\mathrm{TW}}_Q)$ of $U^{\mathrm{TW}}_Q$ fulfils the discrete Euler--Lagrange equation for $L_d^Q$ on the latent space $Q$ up to
\begin{equation}
\left| \sum_{ q_{\mathrm{stencil}} \in \mathcal{D}^Q_\mathrm{TW}} \mathrm{DEL}(L_d^Q)(q_{\mathrm{stencil}})^2\ \right| < 3 \cdot 10^{-9},
\end{equation}
where $\mathcal{D}^Q_\mathrm{TW}$ denotes the triples in $\mathrm{pr}(U^{\mathrm{TW}}_Q)$. However, it is apparently not an accurate approximation of the exact travelling wave (left of \cref{fig:DiscreteWaveTW}).

We conclude that the ROM-based method achieves accurate predictions for initial values close to its training data but can struggles in generalisation tasks as well as in the travelling wave task, in contrast to the proposed stencil based model. This is expected since the latent space was identified based on training data which does not contain travelling waves.

\subsection{Schrödinger equation}

Consider the space-time domain $X=[0,T]  \times [0,b]/\sim$, $T=0.12$, $b=1$ with periodic boundary conditions in space, discretisation parameters $\Delta t = 0.01$, $\Delta x = 0.125$, and linear potential $V(r)= r$. The discrete lattice is denoted by $X_d$ and $M= 1/\Delta x = 8$ is the number of spatial interior grid points and $N = T/\Delta t=12$ the number of time steps.

To obtain training data, we use the 4-point Lagrangian density of \cref{ex:DiscreteSE} and compute 80 solutions: 
identifying $\mathbb{C} \cong \R^2$, $\Psi = \phi + \mathrm i p \mapsto u = \begin{pmatrix}\phi, & p\end{pmatrix}^\top$, the corresponding Lagrangian $L_{\Delta x}^{\Delta t}$ of \eqref{eq:LdelXdelT4} is degenerate and of the form as in \cref{ex:DegenerateL}. Therefore, position data at time $t=0$ is sufficient to initialise the numerical scheme of \cref{rem:comp34Ld} and no velocity data is required. Initial position data is generated based on randomly sampled, exponentially decaying Fourier modes in analogy to \cref{sec:TrainingDataWave}. \Cref{fig:TrainingDataSE} (top left) shows the real part of such a solution. 

\begin{figure}
\includegraphics[width=0.45\linewidth]{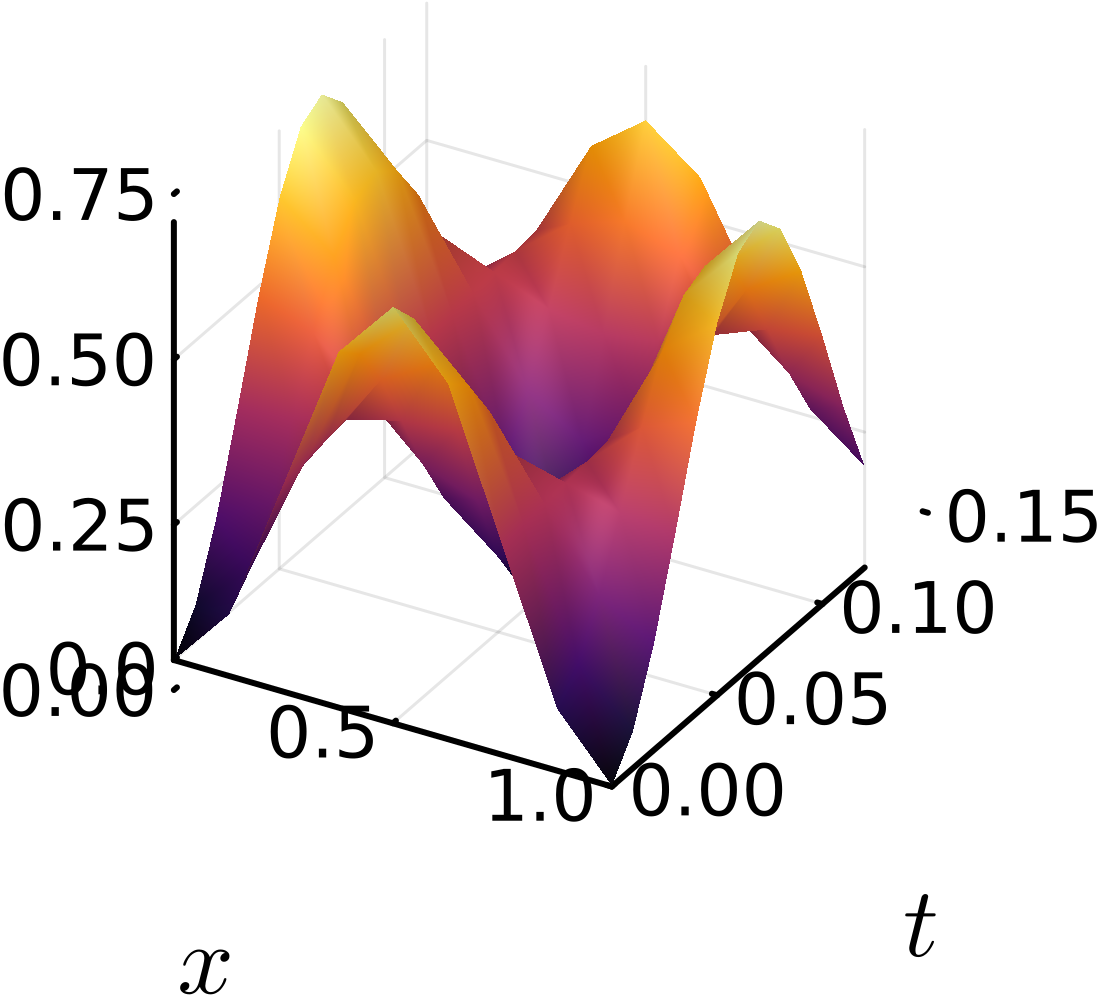}
\includegraphics[width=0.45\linewidth]{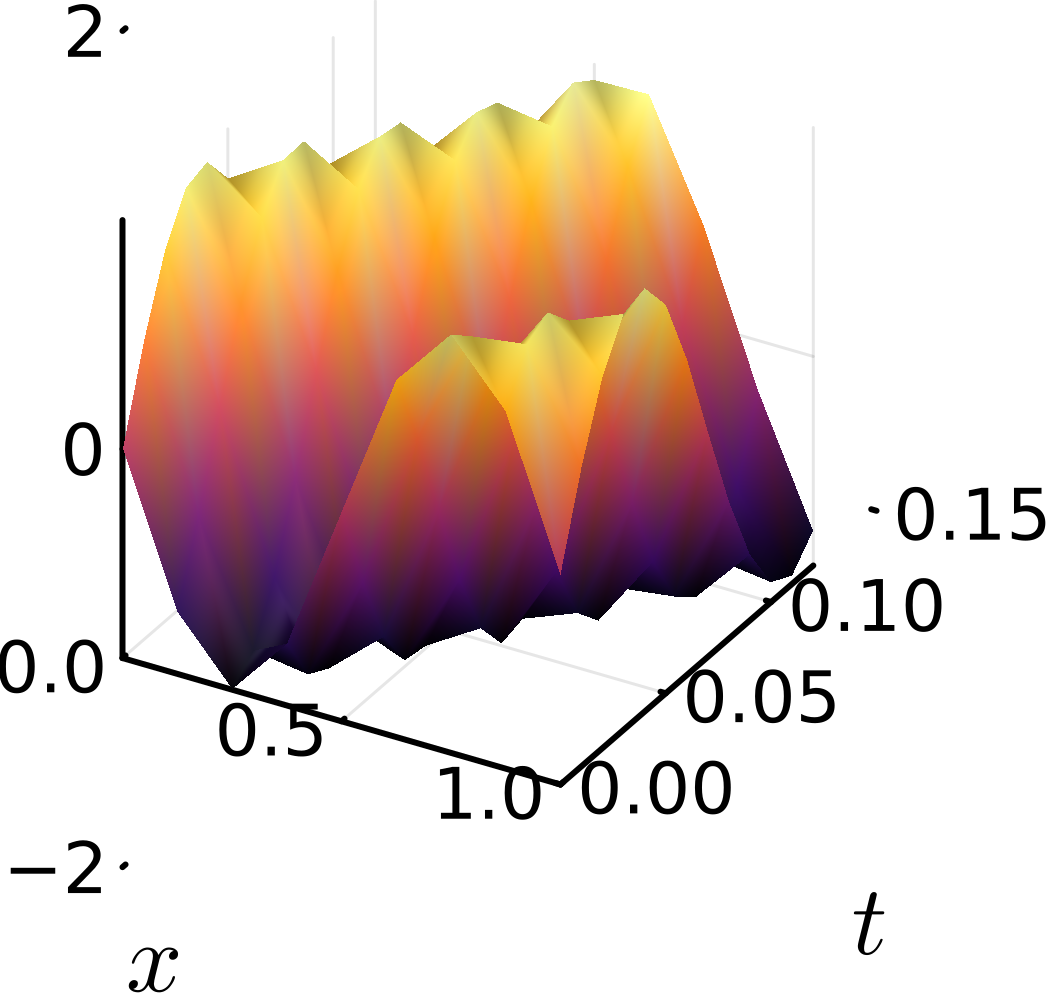}\\
\includegraphics[width=0.45\linewidth]{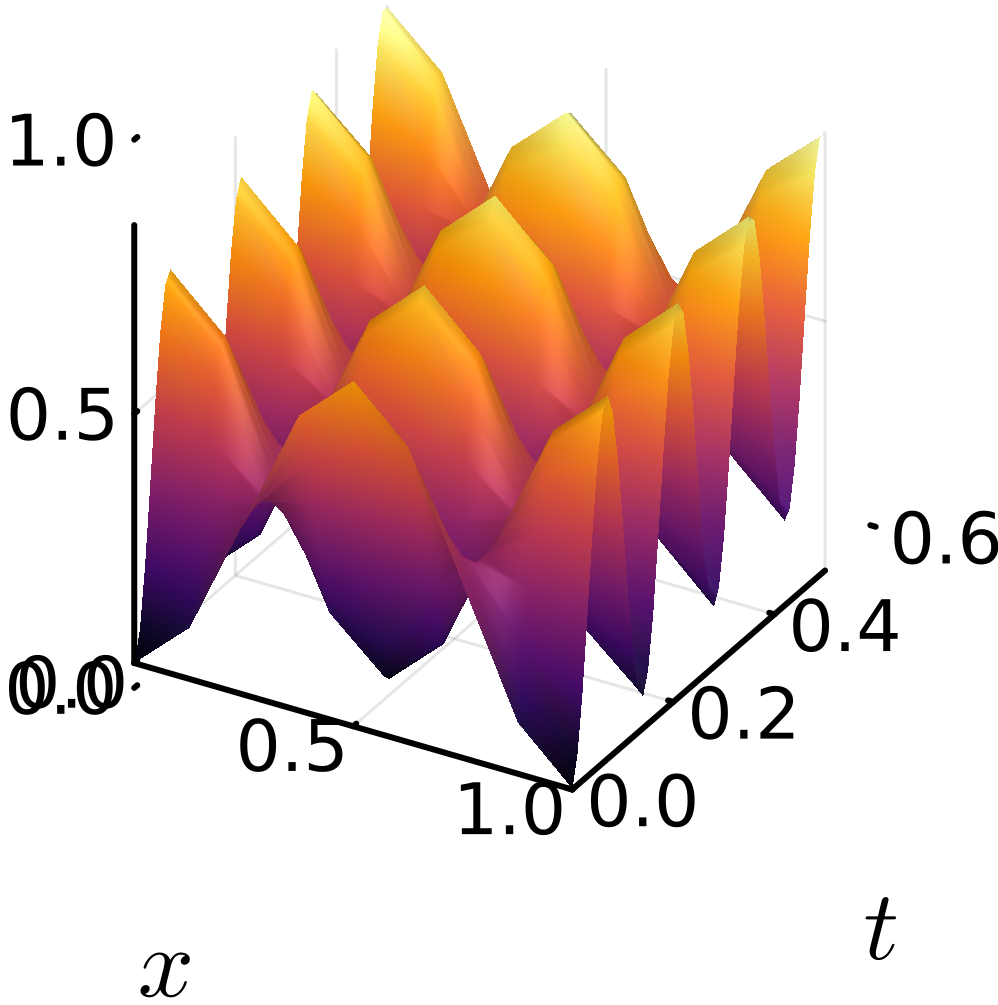}
\includegraphics[width=0.45\linewidth]{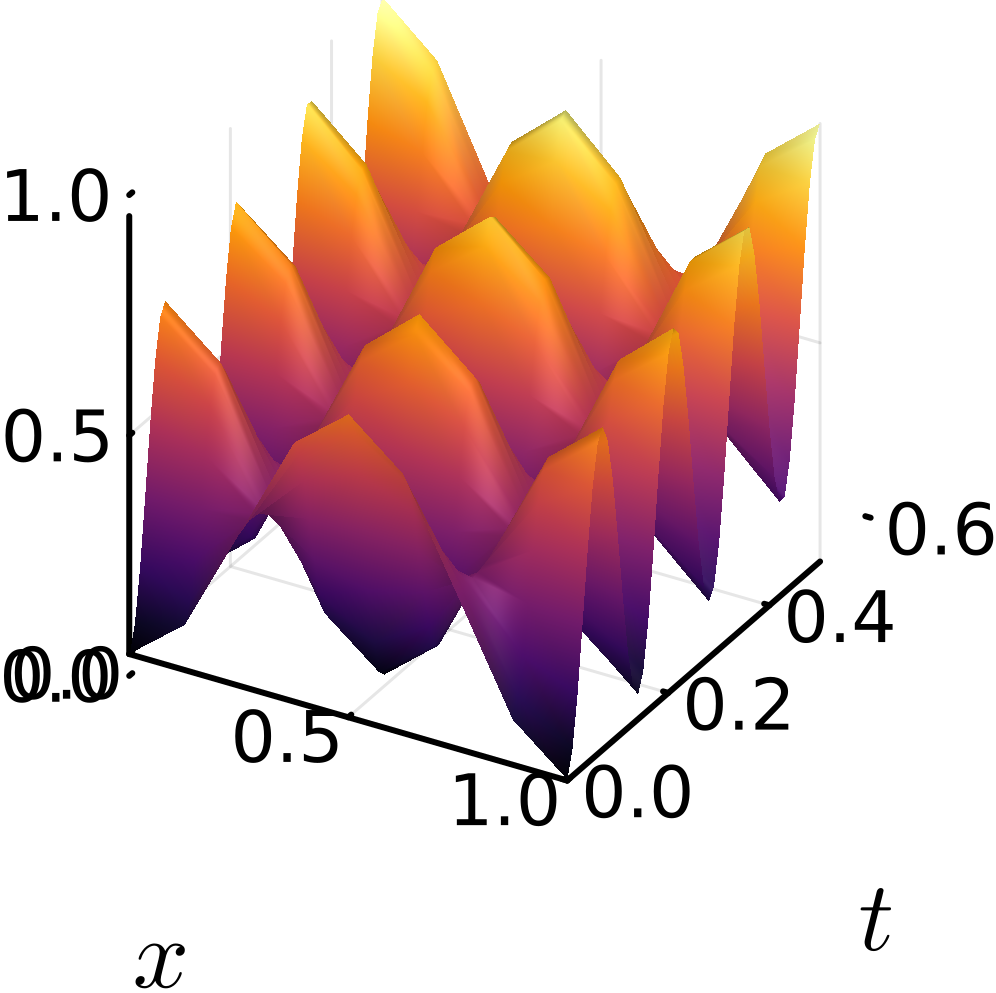}
\caption{Learned discrete Schrödinger equation. Top left: Sample of solution in training data (real part).
Top right: Learned model reproduces travelling wave when initialised from exact travelling wave data (compare to left of \cref{fig:DiscreteWaveTWSE}).
Bottom: Extrapolation of the solution with true discrete Lagrangian (left) and with the learned discrete Lagrangian (right) (error 0.01 in infinity norm).
	}\label{fig:TrainingDataSE}
\end{figure}

A 4-point Lagrangian $L_d \colon (\R^2)^4 \to \R$ (see \cref{ex:3ptLd}) is modelled as a neural network (2 hidden layers of width 12, activation $\mathtt{softplus}$) and fitted to the training data. As a loss function $\ell = \ell_{\mathrm{data}}+\ell_{\mathrm{reg}}$ we adapt the data consistency term $\ell_{\mathrm{data}}$ of \eqref{eq:elldataWave} to 4-point Lagrangians. As regularisation term $\ell_{\mathrm{reg}}$ we use the tamed regulariser \eqref{eq:regLdxdtTamed} to obtain a field theory optimised for temporal forward propagation. 
A loss of $\ell_{\mathrm{data}}<8.04 \cdot 10^{-5}$ and $\ell_{\mathrm{reg}}=0$ is reached after 1544 epochs of batch learning, where a batch consists of 2 blocks and a block contains all stencil data of 3 consecutive time steps. 

\Cref{fig:TrainingDataSE} shows that the trained model successfully extrapolates training data with an extrapolation error $\| |\Psi_\mathrm{true}-\Psi_\mathrm{predicted}|_\mathbb{C}\|_\infty \approx 0.01$ with $T^{\mathrm{extrapolate}} = 50 \Delta t  > 12 \Delta t =T$ (bottom figures).
Here we have used the notation $\| | \Phi |_\mathbb{C}\|_\infty = \max_{i,j} | \Phi^i_{j}|_{\mathbb{C}}$ for a complex matrix $\Phi \in \mathbb{C}^{(N+1)\times M}$.

Moreover, when the learned model is initialised with data $u^0=(u^0_j)_{0\le j < M}$, $u^1=(u^1_j)_{0\le j < M}$ of an exact travelling wave (\cref{rem:DiscreteWaveTWSE}, $m=1$), then it reproduces the exact travelling wave up to an error of $\||\Psi_\mathrm{true}-\Psi_\mathrm{predicted}|_{\mathbb C} \|_\infty \approx 0.19$.
(Compare top right of \cref{fig:TrainingDataSE} showing $\Psi_{\mathrm{predict}}$ with the left of \cref{fig:DiscreteWaveTWSE} showing $\Psi_\mathrm{true}$.)
Notice that travelling wave solutions are captured well in the trained model although the training data does not contain any travelling waves.

\section{Summary and Future Work}\label{sec:Summary}

We have developed a machine learning framework to learn discrete field theories from discrete observation data. Neural network models of discrete Lagrangian densities are fitted such that the discrete Euler Lagrange equations are fulfilled on the training data. 
The discrete Lagrangian densities are based on local stencils motivated by the theory of variational integration methods\cite{MarsdenWestVariationalIntegrators}.
Thus, the local nature of differential equations is reflected in the structure of our approach.

As Lagrangian densities are not uniquely determined by a dynamical system's motions, we present a way to systematically derive regularisation terms for the training procedure. The regularisers guide the optimisation process to a numerical analysis informed model of a discrete Lagrangian density with good convergence properties when used in numerical simulations.

In comparison to architectures that project spatial dimensions to a data-driven latent space (data-driven reduced order models - ROMs) and then learn the temporal dynamics, our framework learns the discrete field theory on a space-time lattice. A reduction of computational complexity of the learning task in our approach is obtained by the use of local stencils rather than by a contraction of dimensionality.
By a comparison with a simple, data-driven ROM based method, we demonstrate that learning on the lattice can be advantageous when highly symmetric solutions such as travelling waves are sought in the data-driven model but are not present in the training data set. Moreover, locality, symmetries, and fundamental principles such as Palais' principle of symmetric criticality can be easier to preserve.

However, while locality of the stencil can be exploited during training in the evaluation of a data consistency term $\ell_{\mathrm{data}}$, depending on boundary conditions, numerical stability concerns, or particularities of an employed stencil, numerical integration of the trained model of a discrete field theory can require solving high-dimensional problems.
Moreover, when the regulariser employed during training is based on the mentioned numerical method, then it can inherit the computational complexity.
While we have shown that one can train the field theory on a coarse mesh while efficiently using the observed training data from the finer mesh, another option is to develop machine learning architectures tailored to numerical integrators developed for high-dimensional problems such as low-rank integrators\cite{WalachThesis,Lubich2018}. This will be the subject of future work.



Further, the machine learning architecture based on discrete Lagrangian densities for local stencils learns structurally simple solutions well and extrapolates correctly to highly symmetric solutions, such as travelling waves, even when these are not in the training data set.
Numerical integration theory offers a diverse set of discretisation methods suitable for a variety of purposes including handling of shock waves or highly oscillatory solutions. In future work, we plan to exploit these methods within the machine learning context to develop reliable architectures for systems with challenging dynamical features.

Additionally,
we propose to learn a representation of a field theory that is as symmetric as possible to improve qualitative features of numerical simulations as we have done in the ode case\cite{SymLNN}. A combination with this technique will allow for the identification of conservation laws of the underlying dynamical system.





\begin{acknowledgments}
C.~Offen acknowledges the Ministerium für Kultur und Wissenschaft des Landes Nordrhein-Westfalen and computing time provided by
the Paderborn Center for Parallel Computing (PC2).
\end{acknowledgments}

\section*{Author declarations}
The authors have no conflicts to disclose.

\section*{Data availability}

The data that support the findings of this study are openly available in the GitHub repository Christian-Offen/DLNN\_pde at
\url{https://github.com/Christian-Offen/DLNN_pde}. An archived version\cite{DLNNPDESoftware} of release v1.0 of the GitHub repository is openly available at \url{https://dx.doi.org/10.5281/zenodo.8245861}.

\appendix

\section{Complementary remarks}\label{sec:Remarks}

\subsection{Gauge freedom}

\begin{Remark}[Gauge freedom - Continuous setting]\label{rem:GaugeFreedom}
	For a continuously differentiable function $F=(F^1,\ldots,F^n)$ in $(x,u)$, denote the total divergence \[
	\nabla_x F(x,u(x)) 
	= \frac{\p }{\p x_0} (F^1(x,u(x)))+\ldots+\frac{\p }{\p x_n} (F^n(x,u(x))).
	\]
	When $\nabla_x F(x,u(x))$ is added to a Lagrangian $L$, then the action $S$ \eqref{eq:ActionFunctionalS} can at most change by a constant, according to the divergence theorem. Therefore, the Euler--Lagrange equations \eqref{eq:EL} are not effected, i.e.\
	\[\mathrm{EL}\left(L+ \nabla_x F(x,u(x))\right) = \mathrm{EL}(L).\]
\end{Remark}

\begin{Remark}[Gauge freedom - Discrete setting]\label{rem:gaugeLd34}
	In analogy to \cref{rem:GaugeFreedom}, discrete Lagrangians are not uniquely determined by a system's discrete motions. Indeed, two 3-point Lagrangians $L_d(a,b,c)$ and $\tilde L_d(a,b,c)$ as in \cref{ex:3ptLd} yield the same discrete Euler--Lagrange equations \eqref{eq:DEL3pt} if they differ by a discrete divergence, i.e.\ if $L_d(a,b,c)-\tilde L_d(a,b,c)$ is of the form
	\begin{equation}\label{eq:gaugeLd3}
		\chi_1(a)-\chi_1(b) + \chi_2(a)-\chi_2(c) +\chi_3(b)-\chi_3(c)
	\end{equation}
	for a differentiable function $\chi = (\chi_1,\chi_2,\chi_3) \colon X \to \R^3$. Analogous observations apply to other types of discrete Lagrangians.
\end{Remark}

\subsection{Use of 7 and 9 point stencil for forward propagations}

Below additional remarks on the computation with the considered 7 and 9 point numerical stencils (\cref{ex:3ptLd}) as described in \cref{rem:comp34Ld}.

\begin{Remark}[Numerical stability of discrete SE]\label{rem:Stability}
	Solutions to the discrete Schrödinger equation \eqref{eq:SEDEL} can be computed as explained in \cref{rem:comp34Ld}. Notice that when the discrete Lagrangian $L_d$ of \eqref{eq:SELd} is expressed in the variable $(u,\phi)$, where $\Psi = u + \mathrm i \phi$, then $L_{\Delta x}^{\Delta t}$ of \cref{rem:comp34Ld} is of the form \eqref{eq:DegenerateL} of \cref{ex:DegenerateL}. Therefore, the computation can be initialised from initial data $(\Psi_{j}^0)_{0\le j\le N_1}$ and boundary conditions $(\Psi_{0}^{i})^{0<i\le N_0}$, $(\Psi_{N_1}^{i})^{0<i\le N_0}$ and is numerically stable\cite{Marsden2002DegenerateLagrangians}. No velocity data is required for the initialisation.
\end{Remark}

\subsection{Numerical approximation of constant in Newton iterations convergence results }\label{sec:ComputeRhoAst}

We provide details on how to approximate the constant $\rho^\ast$ in \cref{prop:SolveDELNewtonU}, which relates to the convergence speed of Newton iterations.

\begin{Remark}\label{rem:periodicBD}
	The matrix $\Lambda=\frac{\p^2 L_{\Delta x}^{\Delta t}}{\p U^{i}\p U^{i+1}}(U^i,U^{i+1})$ of \eqref{eq:rhoU} has the following block-tri-diagonal structure
	\begin{equation}
		\Lambda = \begin{pmatrix}
			A_1&B_1\\
			C_2&A_2&B_2\\
			&\ddots&\ddots&\ddots\\
			&&C_{N_1-2}&A_{N_1-2}&B_{N_1-2}\\
			&&&C_{N_1-1}&A_{N_1-1}\\
		\end{pmatrix} \in (\R^d)^{(N_1-1) \times (N_1-1)},
	\end{equation}
	where
	\begin{equation}\label{eq:ABCLambda}
		\begin{split}
			A_j=\frac{\p^2}{\p u^i_j \p u^{i+1}_j}
			\big(
			&\phantom{+}L_d(u^{i}_{j},u^{i+1}_{j},u^{i}_{j+1},u^{i+1}_{j+1})\\
			&+ L_d(u^{i}_{j-1},u^{i+1}_{j-1},u^{i}_{j},u^{i+1}_{j})
			\big)\\
			B_j=\frac{\p^2}{\p u^i_j \p u^{i+1}_{j+1}}
			\big(
			&\phantom{+}L_d(u^{i}_{j},u^{i+1}_{j},u^{i}_{j+1},u^{i+1}_{j+1})\big)\\
			C_j=\frac{\p^2}{\p u^i_j \p u^{i+1}_{j-1}}
			\big(
			&\phantom{+} L_d(u^{i}_{j-1},u^{i+1}_{j-1},u^{i}_{j},u^{i+1}_{j})
			\big).
		\end{split}
	\end{equation}
	
	If periodic boundary conditions rather than Dirichlet boundary conditions are considered then the spatial mesh has $N_1$ interior mesh points and
	\begin{equation}
		\Lambda = \begin{pmatrix}
			A_0&B_0& & & C_0\\
			C_1&A_1&B_1\\
			&\ddots&\ddots&\ddots\\
			&&C_{N_1-2}&A_{N_1-2}&B_{N_1-2}\\
			B_{N_1-1}&&&C_{N_1-1}&A_{N_1-1}\\
		\end{pmatrix} \in (\R^d)^{N_1 \times N_1}
	\end{equation}
	with $A_j$, $B_j$, $C_j$ as in \eqref{eq:ABCLambda}
\end{Remark}

\begin{Remark}[Numerical computation of $\| \Lambda^{-1}\|$]\label{rem:ApproxRho}
	Let $\| \Lambda \|$ denote the spectral norm of $\Lambda$. The value $\rho^\ast = \| \Lambda^{-1}\|^{-1}$ is the smallest singular value $\sigma$ of $\Lambda$. Its square $\sigma^2$ can be approximated using inverse vector iterations\cite{Deuflhard2003EVChapter} applied to $M=\Lambda^\top \Lambda$: given a random start vector $v'$, iterate
	$v \leftarrow v'$, $v \leftarrow M\backslash v$, $v \leftarrow v/\|v\|$, $\sigma^2 \leftarrow (v^\top v')^{-1}$.
	Here $M\backslash v$ denotes the solution of the linear equation $M x = v$. For this a Cholesky factorisation of $M$ can be computed and used throughout the iteration.
	While in our case sparse structure of $\Lambda$ can be exploited, we refer to the literature\cite{Deuflhard2003EVChapter} for methods that avoid the explicit computation of the matrix product $\Lambda^\top \Lambda$ in the computation of singular values.
	Further, more elaborate tensor representations\cite{RoehringZoellner2022} can be used, especially when generalising to higher-dimensional space-time lattices.
	Approximations to $\sigma^2=(\rho^\ast)^2$ will be used in the machine learning framework introduced in the following sections.
\end{Remark}

\section{Exact discrete Lagrangians}\label{sec:ExactLd}

We can view the various discrete Lagrangians introduced in \cref{sec:DiscreteVarPrinciples} as discretisations of exact discrete Lagrangians. An exact discrete Lagrangian to a continuous field theory has the property that solutions to the discrete Euler--Lagrange equations coincide with the motions of the continuous theory on the grid. 
For this we extend our definition of discrete Lagrangians from \cref{sec:DiscreteVarPrinciples}. For a more detailed introduction and further information we refer to \cite{Marsden2001}.

With the notation of \cref{sec:DiscreteVarPrinciples}, let $\p X^l$ denote the boundary of the hypercube $X^l$. Let $\p X_d = \bigcup_{l \in \mathcal I} \p X^l \subset \R^{n+1}$ denote the $n$-skeleton of $X$.

\begin{Definition}[Discrete Lagrangian and action]\label{def:Sd}
	A {\em discrete Lagrangian} $L_d \colon \bigcup_l (\p X^l \times \R^k) \to \R$ maps any graph of a boundary function $u_{\p X^l} \colon \p X^l \to \R^k$ to a real number. The corresponding {\em discrete action} $S_d \colon \p X_d \times \R^k \to \R$ maps any graph of a boundary function $u_{\p X_d}\colon \p X_d \to \R^k$ on the $n$-skeleton to
	\begin{equation}\label{eq:DiscreteS}
		S_d(u_{\p X_d}) = \sum_{l \in \mathcal I} L_d(u_{\p X_d}|_{\p X^l}).
	\end{equation}
	Here $u_{\p X_d}|_{\p X^l}$ denotes the restriction of $u_{\p X_d}$ to the boundary $\p X^l$ of the cube $X^l$.	
\end{Definition}

Notice that in \cref{def:Sd} $L_d$ is allowed to depend on the graph of the boundary function rather than just its values, which is implicit in the notation.

\begin{Definition}[Discrete Euler--Lagrange equations]\label{def:DEL}
	Let $L_d$ be a discrete Lagrangian. For each interior vertex $v \in \mathring X_d$ let $\mathcal{I}_v \subset \mathcal{I}$ contain the indices of hypercubes $X^l$ containing $v$. The {\em discrete Euler--Lagrange equations} are given as
	\begin{equation}\label{eq:DEL}
		\sum_{l \in I_v} \frac{\delta L_d}{\delta u_{\p X^l}} (u_{\p X_d}|_{\p X^l}) =0
	\end{equation}
	for $u_{\p X_d} \colon \p X_d \to \R^k$. Here $ \frac{\delta L_d}{\delta u_{\p X^l}}$ denotes the variational derivative.
\end{Definition}

\begin{Example}[Exact discrete Lagrangian]
	For a Lagrangian $L$ as in \eqref{eq:ActionFunctionalS} we define the {\em exact discrete Lagrangian}
	\begin{equation}\label{eq:ExactLd}
		L_d^{\mathrm{exact}} (u_{\p X^l}) = \int_{X^l} L(x,u(x),u_{x_0}(x),\ldots,u_{x_n}(x)) \d x,
	\end{equation}
	where $u \colon X^l \to \R^k$ is the solution (assuming existence/uniqueness) of the Euler--Lagrange equation $\mathrm{EL}(L) =0$ subject to the boundary condition $u|_{\p X^l} = u_{\p X^l}$.
\end{Example}


We conclude with the following observations.
	
	\begin{itemize}
		\item If $u \colon X \to \R^k$ solves the Euler--Lagrange equation \eqref{eq:EL} then $S(u)=S_d^{\mathrm{exact}}(u|_{\p X_d})$, when $S_d$ is the discrete action \eqref{eq:DiscreteS} for the exact discrete Lagrangian $L_d = L_d^{\mathrm{exact}}$.
		
		\item 
		In the ode case $\dim X=1$ we have $\p X_d = X_d = \{ x_0^0,\ldots,x_0^{N_0}\}$ and the exact discrete action is defined for functions on the mesh $X_d$.
		
		\item
		The discrete Lagrangians \eqref{eq:LdWave2d} of \cref{ex:DiscreteWave} and \eqref{eq:SELd} of \cref{ex:DiscreteSE} can be obtained from \eqref{eq:ExactLd} by quadrature.
		
	\end{itemize}
	

\section{Convergence of Newton iterations to solve the 7-point stencil}\label{sec:ProofConvergence7Stencil}

We provide a proof of \cref{prop:SolveDELNewton}.

\begin{Proposition*}\label{prop:SolveDELNewtonAppendix}
	Let $u^{i}_{j}$, $u^{i+1}_{j}$, $u^{i}_{j+1}$, $u^{i-1}_{j}$, $u^{i-1}_{j+1}$, $u^{i}_{j-1}$, $u^{i+1}_{j-1}$ such that the 7-point stencil \eqref{eq:DEL3pt} is fulfilled. Let $\mathcal O\subset \R^d$ be a convex neighbourhood of $u^\ast=u^{i+1}_{j}$, $\| \cdot \|$ a norm of $\R^d$ inducing an operator norm on $\R^{d \times d}$.
	Define
	\[p(u) :=\frac{\p^2 L_d}{\p u^{i}_{j}\p u}(u^i_j,u,u^i_{j+1}) \in \R^{d \times d}\]
	and let $\theta$ and $\overline{\theta}$ be Lipschitz constants on $\mathcal O$ for $p$
	and for $\mathrm{inv} \circ p$, respectively, where $\mathrm{inv}$ denotes matrix inversion.
	Let
	\begin{equation}\label{eq:rhoAppendix}
		\rho^\ast := \left\|\mathrm{inv}(p(u^\ast))\right\|
		= \left\|\left(\frac{\p^2 L_d}{\p u^{i}_{j}\p u^\ast}(u^i_j,u^\ast,u^i_{j+1})\right)^{-1}\right\|
	\end{equation}
	and let $f(u^{(n)})$ denote the left hand side of \eqref{eq:DEL3pt} with $u_j^{i+1}$ replaced by $u^{(n)}$.
	If $\|{u}^{(0)}-{u}^\ast\| \le \min\left(\frac{\rho^\ast}{\overline{\theta}}, \frac{1}{2\theta\rho^\ast}\right)$ for ${u}^{(0)}\in \mathcal{O}$, then the Newton iterations
	\[{u}^{(n+1)}:= {u}^{(n)} - \mathrm{inv}(p(u^{(n)})) f(u^{(n)})\]
	converge quadratically against ${u}^{\ast}$, i.e.\
	\begin{equation}\label{eq:NewtonQuadraticConvergenceAppendix}
		\|{u}^{(n+1)} - {u}^{\ast} \| \le \rho^\ast \theta \|{u}^{(n)} - {u}^{\ast} \|^2. 
	\end{equation}
	
\end{Proposition*}

\begin{proof}
	We adapt standard estimates for Newton's method (see \cite[\S 4]{Deuflhard2003}, for instance) to the considered setting.
	Let $f \colon \mathcal{O} \to \R^d$ with
	\begin{equation*}
		\begin{split}
			f(u)=\frac{\p}{\p u^i_j}
			\big(
			&\phantom{+}L_d(u^{i}_{j},u,u^{i}_{j+1})
			+ L_d(u^{i-1}_{j},u^{i}_{j},u^{i-1}_{j+1})\\
			&+ L_d(u^{i}_{j-1},u^{i+1}_{j-1},u^{i}_{j})
			\big)
		\end{split}
	\end{equation*}
	With this definition, $f(u^\ast)=0$, $\theta$ is a Lipschitz constant for $\D f \colon \mathcal{O} \to \R^{d \times d}$, $u \mapsto \D f(u)$, $\overline{\theta}$ is a Lipschitz constant for $\mathrm{inv} \circ \D f \colon \mathcal{O} \to \R^{d \times d}$, $u \mapsto \D f(u)^{-1}$, and $\rho^\ast = \| \D f(u^\ast)^{-1} \|$.
	Here $\D f(u)$ denotes the Jacobian matrix of $f$ at $u \in \mathcal O$.
	
	Assume that for $n \in \N$ an iterate $u^{(n)} \in \mathcal{O}$ fulfils $\|u^{(n)}-u^\ast\| \le \min\left(\frac{\rho^\ast}{\overline{\theta}}, \frac{1}{2\theta\rho^\ast}\right)$. Then \begin{equation*}\begin{split}
			\| \D f(u^{(n)})^{-1} &\| 
			= \| \D f(u^{(n)})^{-1} - \D f(u^{\ast})^{-1} + \D f(u^{\ast})^{-1}\| \\
			&\le \| \D f(u^{(n)})^{-1} - \D f(u^{\ast})^{-1} \| + \rho^\ast \\
			&\le \underbrace{\overline{\theta} \|u^{(n)} - u^\ast \|}_{\le \rho^\ast} + \rho^\ast \le 2 \rho^\ast.
	\end{split}\end{equation*}
	For the next iterate $u^{(n+1)} = u^{(n)} - \D f(u^{(n)})^{-1} f(u^{(n)})$ the distance to $u^\ast$ can be bounded:

	\begin{align*}
		\| &u^{(n+1)} - u^\ast \| \\
		&= \| u^{(n)} - u^\ast - \D f(u^{(n)})^{-1} (f(u^{(n)}) - \underbrace{f(u^\ast)}_{=0})  \| \\
		&\le \underbrace{\|\D f(u^{(n)})^{-1}  \|}_{\le 2 \rho^\ast}
		\Big\| \underbrace{\D f(u^{(n)})(u^{(n)}-u^\ast)}_{=\int_0^1 \D f(u^{(n)})(u^{(n)}-u^\ast) \d t }\\
		&\phantom{\|\D f(u^{(n)})^{-1}  \|}- \underbrace{(f(u^{(n)}) - f(u^\ast) )}_{\int_0^1 \D f (u^{(n)} + t (u^\ast -u^{(n)}))(u^{(n)} - u^\ast) \d t }  \Big\| \\
		&\le 2 \rho^\ast \Big\| \int_0^1 \left(\D f(u^{(n)}) - \D f(u^{(n)} + t(u^\ast - u^{(n)}))\right)\\
		&\qquad \qquad \qquad \qquad \qquad \qquad \qquad \qquad  \cdot (u^{(n)}-u^\ast) \d t \Big\|\\
		&\le 2 \rho^\ast  \int_0^1 \underbrace{\left\|\D f(u^{(n)}) - \D f(u^{(n)} + t(u^\ast - u^{(n)}))\right\|}_{\le \theta t \|u^{(n)}-u^\ast\| }\\
		&\qquad \qquad \qquad \qquad \qquad \qquad \qquad \qquad 
		\cdot \left\|u^{(n)}-u^\ast\right\| \d t\\
		&\le 2 \rho^\ast \theta \left\|u^{(n)}-u^\ast\right\|^2  \int_0^1 t \d t\\
		&=\rho^\ast \theta \left\|u^{(n)}-u^\ast\right\|^2
	\end{align*}
	
	Moreover, since $\left\|u^{(n)}-u^\ast\right\| \le  (2\theta \rho^\ast)^{-1}$, we have
	$\|u^{(n+1)} - u^\ast \| \le  \frac 12 \|u^{(n)} - u^\ast \|$. By induction, the Newton iterations converge against $u^\ast$ and \eqref{eq:NewtonQuadraticConvergenceAppendix} holds true.
\end{proof}

\section{Symmetric criticality of travelling waves in the wave equation}\label{sec:ProofPalaisTW}

We provide a proof of \cref{prop:PalaisTW}.

\begin{Proposition*}[PSC for travelling waves]
Let $c \in \R \setminus \{0\}$, $b>0$, $X=[0,b/c]/\sim \times [0,b]/\sim$ be a torus, $W=C^1(X,\R^d)$ continuously differentiable functions and $\Sigma = \{u \in W \,|\, u(t+s,x+cs)=u(t,x)\}$ be travelling waves with wave speed $c$.
Let $S \colon W \to \R$
\begin{equation}\label{eq:SPSCAppendix}
S(u)=\int_0^{b/c}\int_0^b L(u(t,x),u_t(t,x),u_x(t,x)) \d x \d t
\end{equation}
be a continuously differentiable action functional with an autonomous Lagrangian $L \colon (\R^d)^3 \to \R$. 
Denote the restriction of $S$ to $\Sigma$ by $S_\Sigma \colon \Sigma \to \R$.
Then a travelling wave $u \in \Sigma$ is a stationary point of $S$, if and only if $u$ is a stationary point of $S_\Sigma$. Using the identification $\Sigma \cong C^1([0,b]/\sim,\R^d)$, $u(t,x)=f(x-ct)=f(\xi)$ for $u \in \Sigma$, we have
\begin{equation}\label{eq:SSigmaLAppendix}
S_\Sigma(f) = \frac bc \int_0^b L(f(\xi),-c f_\xi(\xi),f_\xi(\xi)) \d \xi.
\end{equation}
\end{Proposition*}

\begin{proof}
Equipped with the norm $\| \cdot \|_W$ with $\|u\|_W = \max_{(t,x)\in X} |u(t,x)| + \max_{(t,x)\in X} \| D u(t,x)\|_{\R^d \times \R^2}$ the space $W$ is a Banach space since the torus $X$ is compact.
The group $([0,b]/\sim,+)$ is compact and acts on the Banach space $W$ by $u \mapsto s.u$ with $s.u(t,x)=u(t+s,x+cs)$. The fixpoints of the group action are $\Sigma$. The considered action $S$ is invariant under the action, i.e.\ $S(u)=S(s.u)$ for all $s \in [0,b]/\sim$ since the Lagrangian is autonomous: let $u \in W$, $s \in [0,b)$.
\begin{equation}\label{eq:ComputationPropPalaisTW}
	\begin{split}
S(s.u) &= \int_0^{b/c}\int_0^b L(s.u(t,x),(s.u)_t(t,x),(s.u)_x(t,x)) \d x \d t\\
&=\int_{-s/c}^{(b-s)/c}\int_{-s}^{b-s} L(u(t,x),u_t(t,x),u_x(t,x)) \d x \d t\\
&=\int_0^{b/c}\int_0^b L(u(t,x),u_t(t,x),u_x(t,x)) \d x \d t\\
&=S(u).
	\end{split}
\end{equation}
By compactness of the group $([0,b]/\sim,+)$, Palais' principle of symmetric criticality\cite{palais1979} guarantees that
critical points of the restricted functional $S_\Sigma \colon \Sigma \to \R$ constitute critical points of $S$.

To proof validity of \eqref{eq:SSigmaLAppendix}, consider coordinates $(x-ct,x+ct) = (\xi,\theta)$ and a travelling wave $u(t,x)=f(x-ct)=f(\xi)$, $f \colon [0,b]/\sim \to \R^d$. We have
\begin{equation}\label{eq:SPSCL}
	\begin{split}
S_\Sigma(f)
&= S(u)
=\int_0^{b/c}\int_0^b L(f(\xi),-c f_\xi(\xi),f_\xi(\xi))\d x \d t\\
&=\int_{-b}^b\int_{|\xi|}^{2b-|\xi|} L(f(\xi),-c f_\xi(\xi),f_\xi(\xi)) \frac{\d\theta  \d \xi }{2c}\\
&\stackrel{(\ast)}{=}\int_{0}^b\int_{0}^{2b} L(f(\xi),-c f_\xi(\xi),f_\xi(\xi)) \frac{\d\theta  \d \xi }{2c}\\
&=\frac b c \int_{0}^b L(f(\xi),-c f_\xi(\xi),f_\xi(\xi)) \d \xi
	\end{split},
\end{equation}
where we have used periodicity $f(\xi + b)=f(\xi)$ in $(\ast)$.
\end{proof}

\begin{Remark}
For standing travelling waves $c=0$, the torus $X=[0,1]/\sim \times [0,b]/\sim$ is considered. An analogous statement is obtained with $b/c$ substituted by 1 in \eqref{eq:SPSCAppendix} and \eqref{eq:SSigmaLAppendix}.
\end{Remark}

\bibliography{resources}

\end{document}